\documentclass[reqno,12pt,a4paper]{amsart}

\usepackage{amsmath}
\usepackage{amssymb}
\usepackage{amsfonts}
\usepackage{latexsym}
\usepackage{amsthm}

\usepackage{bm}

\usepackage{graphicx}

\begin{document}


\renewcommand{\theequation}{\arabic{section}.\arabic{equation}}
\theoremstyle{plain}
\newtheorem{theorem}{\bf Theorem}[section]
\newtheorem{lemma}[theorem]{\bf Lemma}
\newtheorem{corollary}[theorem]{\bf Corollary}
\newtheorem{proposition}[theorem]{\bf Proposition}
\newtheorem{definition}[theorem]{\bf Definition}
\newtheorem{remark}[theorem]{\bf Remark}

\def\a{\alpha}  \def\cA{{\mathcal A}}     \def\bA{{\bf A}}  \def\mA{{\mathscr A}}
\def\b{\beta}   \def\cB{{\mathcal B}}     \def\bB{{\bf B}}  \def\mB{{\mathscr B}}
\def\g{\gamma}  \def\cC{{\mathcal C}}     \def\bC{{\bf C}}  \def\mC{{\mathscr C}}
\def\G{\Gamma}  \def\cD{{\mathcal D}}     \def\bD{{\bf D}}  \def\mD{{\mathscr D}}
\def\d{\delta}  \def\cE{{\mathcal E}}     \def\bE{{\bf E}}  \def\mE{{\mathscr E}}
\def\D{\Delta}  \def\cF{{\mathcal F}}     \def\bF{{\bf F}}  \def\mF{{\mathscr F}}
\def\c{\chi}    \def\cG{{\mathcal G}}     \def\bG{{\bf G}}  \def\mG{{\mathscr G}}
\def\z{\zeta}   \def\cH{{\mathcal H}}     \def\bH{{\bf H}}  \def\mH{{\mathscr H}}
\def\e{\eta}    \def\cI{{\mathcal I}}     \def\bI{{\bf I}}  \def\mI{{\mathscr I}}
\def\p{\psi}    \def\cJ{{\mathcal J}}     \def\bJ{{\bf J}}  \def\mJ{{\mathscr J}}
\def\vT{\Theta} \def\cK{{\mathcal K}}     \def\bK{{\bf K}}  \def\mK{{\mathscr K}}
\def\k{\kappa}  \def\cL{{\mathcal L}}     \def\bL{{\bf L}}  \def\mL{{\mathscr L}}
\def\l{\lambda} \def\cM{{\mathcal M}}     \def\bM{{\bf M}}  \def\mM{{\mathscr M}}
\def\L{\Lambda} \def\cN{{\mathcal N}}     \def\bN{{\bf N}}  \def\mN{{\mathscr N}}
\def\m{\mu}     \def\cO{{\mathcal O}}     \def\bO{{\bf O}}  \def\mO{{\mathscr O}}
\def\n{\nu}     \def\cP{{\mathcal P}}     \def\bP{{\bf P}}  \def\mP{{\mathscr P}}
\def\r{\rho}    \def\cQ{{\mathcal Q}}     \def\bQ{{\bf Q}}  \def\mQ{{\mathscr Q}}
\def\s{\sigma}  \def\cR{{\mathcal R}}     \def\bR{{\bf R}}  \def\mR{{\mathscr R}}
                \def\cS{{\mathcal S}}     \def\bS{{\bf S}}  \def\mS{{\mathscr S}}
\def\t{\tau}    \def\cT{{\mathcal T}}     \def\bT{{\bf T}}  \def\mT{{\mathscr T}}
\def\f{\phi}    \def\cU{{\mathcal U}}     \def\bU{{\bf U}}  \def\mU{{\mathscr U}}
\def\F{\Phi}    \def\cV{{\mathcal V}}     \def\bV{{\bf V}}  \def\mV{{\mathscr V}}
\def\P{\Psi}    \def\cW{{\mathcal W}}     \def\bW{{\bf W}}  \def\mW{{\mathscr W}}
\def\o{\omega}  \def\cX{{\mathcal X}}     \def\bX{{\bf X}}  \def\mX{{\mathscr X}}
\def\x{\xi}     \def\cY{{\mathcal Y}}     \def\bY{{\bf Y}}  \def\mY{{\mathscr Y}}
\def\X{\Xi}     \def\cZ{{\mathcal Z}}     \def\bZ{{\bf Z}}  \def\mZ{{\mathscr Z}}
\def\O{\Omega}

\def\ve{\varepsilon}   \def\vt{\vartheta}    \def\vp{\varphi}    \def\vk{\varkappa}

\def\Z{{\mathbb Z}}    \def\R{{\mathbb R}}   \def\C{{\mathbb C}}
\def\T{{\mathbb T}}    \def\N{{\mathbb N}}   \def\dD{{\mathbb D}}
\def\H{{\mathbb H}}


\def\la{\leftarrow}              \def\ra{\rightarrow}            \def\Ra{\Rightarrow}
\def\ua{\uparrow}                \def\da{\downarrow}
\def\lra{\leftrightarrow}        \def\Lra{\Leftrightarrow}
\def\rra{\rightrightarrows}


\def\lt{\biggl}                  \def\rt{\biggr}
\def\ol{\overline}               \def\wt{\widetilde}


\let\ge\geqslant                 \let\le\leqslant
\def\lan{\langle}                \def\ran{\rangle}
\def\/{\over}                    \def\iy{\infty}
\def\sm{\setminus}               \def\es{\emptyset}
\def\ss{\subset}                 \def\ts{\times}
\def\pa{\partial}                \def\os{\oplus}
\def\om{\ominus}                 \def\ev{\equiv}
\def\iint{\int\!\!\!\int}        \def\iintt{\mathop{\int\!\!\int\!\!\dots\!\!\int}\limits}
\def\el2{\ell^{\,2}}             \def\1{1\!\!1}
\def\sh{\sharp}


\def\Area{\mathop{\mathrm{Area}}\nolimits}
\def\arg{\mathop{\mathrm{arg}}\nolimits}
\def\const{\mathop{\mathrm{const}}\nolimits}
\def\det{\mathop{\mathrm{det}}\nolimits}
\def\diag{\mathop{\mathrm{diag}}\nolimits}
\def\diam{\mathop{\mathrm{diam}}\nolimits}
\def\dim{\mathop{\mathrm{dim}}\nolimits}
\def\dist{\mathop{\mathrm{dist}}\nolimits}
\def\Im{\mathop{\mathrm{Im}}\nolimits}
\def\Int{\mathop{\mathrm{Int}}\nolimits}
\def\Ker{\mathop{\mathrm{Ker}}\nolimits}
\def\Length{\mathop{\mathrm{Length}}\nolimits}
\def\Lip{\mathop{\mathrm{Lip}}\nolimits}
\def\rank{\mathop{\mathrm{rank}}\limits}
\def\Ran{\mathop{\mathrm{Ran}}\nolimits}
\def\Re{\mathop{\mathrm{Re}}\nolimits}
\def\Res{\mathop{\mathrm{Res}}\nolimits}
\def\res{\mathop{\mathrm{res}}\limits}
\def\sign{\mathop{\mathrm{sign}}\nolimits}
\def\supp{\mathop{\mathrm{supp}}\nolimits}
\def\Tr{\mathop{\mathrm{Tr}}\nolimits}

\mathchardef\gRe="023C
\mathchardef\gIm="023D


\def\SLE{\mathrm{SLE}}
\def\DS{\diamondsuit}
\def\CaraTo{\mathop{\longrightarrow}\limits^{\mathrm{Cara}}}



\def\mesh{\delta}
\newcommand\weightG[1]{\mu^\mesh_\G(#1)}
\newcommand\weightE[2]{{\mu_{#1#2}}}
\newcommand\weightDS[1]{\mu^\mesh_\DS(#1)}
\newcommand\weightL[1]{\mu^\mesh_\L(#1)}
\newcommand\dhm[4]{\o^{#1}(#2;#3;#4)}
\renewcommand\hm[3]{\o(#1;#2;#3)}
\def\dpa{\partial^\mesh}
\def\dopa{\overline{\partial}\vphantom{\partial}^\mesh}
\def\sps{($\spadesuit$)}

\def\ccdot{\,\cdot\,}

\def\Pr{\mathrm{Proj}}

\title{Discrete complex analysis on isoradial graphs}

\date{\today}
\author[Dmitry Chelkak]{Dmitry Chelkak$^\mathrm{a,c}$}

\author[Stanislav Smirnov]{Stanislav Smirnov$^\mathrm{b,c}$}

\thanks{\textsc{${}^\mathrm{A}$ St.Petersburg Department of Steklov Mathematical Institute (PDMI RAS).
Fontanka~27, 191023 St.Petersburg, Russia.}}

\thanks{\textsc{${}^\mathrm{B}$ Section de Math\'ematiques, Universit\'e de Gen\`eve.
2-4 rue du Li\`evre, Case postale~64, 1211 Gen\`eve 4, Suisse.}}

\thanks{\textsc{${}^\mathrm{C}$ Chebyshev Laboratory, Department of Mathematics and
Mechanics, Saint-Petersburg State University. 14th Line, 29b, 199178 Saint-Petersburg,
Russia.}}

\thanks{{\it E-mail addresses:} \texttt{dchelkak@pdmi.ras.ru, Stanislav.Smirnov@unige.ch}}

\begin{abstract}
We study discrete complex analysis
and potential theory
on a large family of planar graphs,
the so-called isoradial ones.
Along with discrete analogues of several
classical results,
we prove uniform convergence of
discrete harmonic measures, Green's functions and
Poisson kernels to their continuous counterparts.
Among other applications, the results can be used
to establish
universality of the critical Ising and other lattice models.
\end{abstract}



\maketitle

\section{Introduction}

\subsection{Motivation}

This paper is concerned with discrete versions
of complex analysis and potential theory in the complex plane.
There are many discretizations of harmonic and holomorphic functions,
which have a long history.
Besides proving discrete analogues of the usual complex analysis
theorems, one can ask to which extent
discrete objects approximate their continuous counterparts.
This can be used to give ``discrete'' proofs of continuous theorems
(see, e.g., \cite{L-F55} for such a proof of the Riemann mapping theorem)
or to prove convergence of discrete objects to continuous ones.
One of the goals of our paper is to provide tools
for establishing convergence of 
critical 2D lattice models to conformally invariant scaling limits.

There are no ``canonical'' discretizations of Laplace and Cauchy-Riemann operators, the most
studied ones  (and perhaps the most convenient) are for the square grid. There are also
definitions for other regular lattices, as well as generalizations to larger families of
embedded into $\C$ planar graphs (see \cite{Sm10} and references therein).

We will work with isoradial graphs (or, equivalently, rhombic lattices) where all faces can be
inscribed into circles of equal radii. Rhombic lattices were introduced by R.~J.~Duffin
\cite{Duf68} in late sixties as (perhaps) the largest family of graphs for which the
Cauchy-Riemann operator admits a nice discretization, similar to that for the square lattice.
They reappeared recently as isoradial graphs in the work of Ch.~Mercat \cite{Mer01} and
R.~Kenyon \cite{Ken02}, as the largest family of graphs where certain 2D statistical
mechanical models (notably the Ising and dimer models) preserve some integrability properties.
Note that isoradial graphs can be quite irregular -- see e.g. Fig.~\ref{Fig:IsoGraph}A. It was
shown by R.~Kenyon and J.-M.~Schlenker \cite{KSch04} that many planar graphs admit isoradial
embeddings -- in fact, there are only two topological obstructions. Also isoradial graphs have
a well-defined mesh size $\mesh$ -- the common radius of the circumscribed circles.

It is thus natural to consider this family of graphs
in the context of {\em universality} for 2D models with (conjecturally) conformally invariant scaling limits
(as the mesh tends to zero).

The primary goal of our paper is to provide a ``toolbox'' of discrete versions of continuous
results (particularly ``hard'' estimates) sufficient to perform a passage to the scaling
limit. Of particular interest to us is the {\em critical Ising model}, and this paper starts a
series devoted to its {\em universality} (which means that the scaling limit is independent of
the shape of the lattice). See \cite{Sm06}, \cite{ChSm08} for the strategy of our proof,
\cite{ChSm09} for the convergence of certain discrete holomorphic observables and \cite{Sm07}
for the square lattice case.

Our results can also be applied to other lattice models.
The uniform convergence of the discrete Poisson kernel (\ref{DefDiscrPoisson})
already implies {universality} for the 
loop-erased random walks on isoradial graphs.
Namely, our paper together with \cite{LSchW04} implies that
their trajectories converge to SLE$(2)$
curves (see Sect.~3.2, especially Remark~3.6, in \cite{LSchW04}).
There are several other fields where discrete harmonic and
discrete holomorphic functions defined on isoradial graphs play essential role and hence where
our results may be useful: approximation of conformal maps \cite{Buck08}; discrete integrable
systems \cite{BMS05}; and the theory of discrete Riemann surfaces \cite{Mer07}.

Local convergence of discrete harmonic (holomorphic) functions to continuous harmonic
(holomorphic) functions is a rather simple fact. Moreover, it was shown by Ch.~Mercat
\cite{Mer02} that each continuous holomorphic function can be approximated by discrete ones.
Thus, the discrete theory is close to the continuous theory ``locally.'' Nevertheless, until
recently almost nothing was known about the ``global''\ convergence of the functions defined
in discrete domains as the solutions of some discrete {\it boundary value problems} to their
continuous counterparts. This setup goes back to the seminal paper by R.~Courant,
K.~Friedrichs and H.~Lewy \cite{CFL28}, where convergence is established for harmonic
functions with smooth Dirichlet boundary conditions in smooth domains, discretized by the
square lattice, but not much progress has occurred since. For us it is important to consider
discrete domains with possibly very {\it rough boundaries} and to establish convergence
without any regularity assumptions about them. Besides being of independent interest, this is
indispensable for establishing convergence to Oded Schramm's SLEs, since the latter curves are
fractal.
\begin{figure}
\centering{
\begin{minipage}[c]{0.6\textwidth}
\centering{\includegraphics[angle=90, height=0.35\textheight]{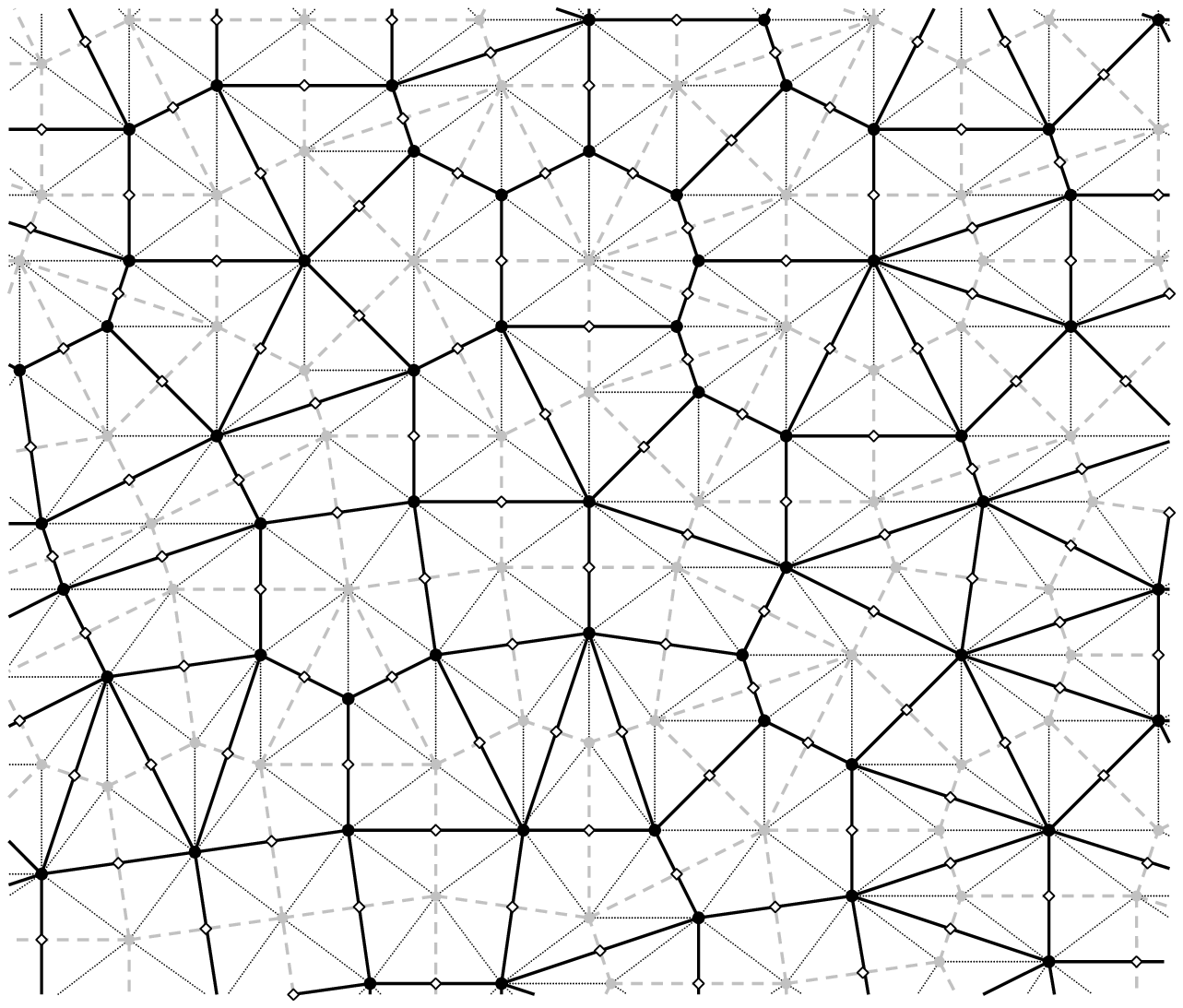}}

\centering{\textsc{(A)}}
\end{minipage}
\hskip 0.05\textwidth
\begin{minipage}[c]{0.34\textwidth}
\vspace{0.05\textheight}

\centering{\includegraphics[height=0.25\textheight]{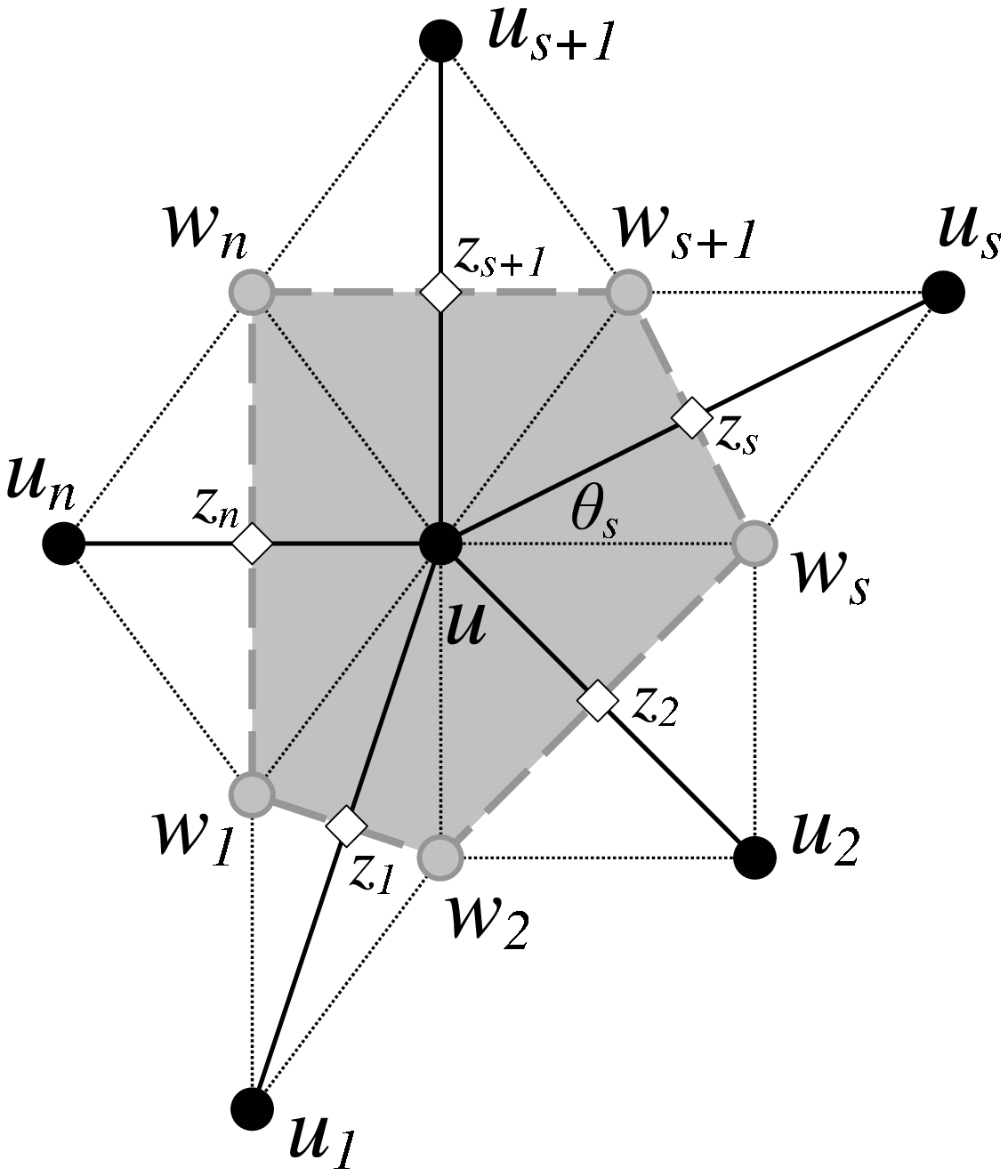}}

\vspace{0.05\textheight}

\centering{\textsc{(B)}}
\end{minipage}
}\caption{\textsc{(A)}\label{Fig:IsoGraph} An isoradial graph $\G$ (black vertices, solid
lines), its dual isoradial graph $\G^*$ (gray vertices, dashed lines), the corresponding
rhombic lattice or quad-graph (vertices $\L=\G\cup\G^*$, thin lines, rhombic faces) and the
set $\DS=\L^*$ of rhombi centers (diamond-shaped points). \textsc{(B)} Local notations near
$u\in\G$. The dual face $W(u)$ is shaded.\label{Fig:Notations}}
\end{figure}

\subsection{Preliminary definitions.}
The planar graph $\G$ embedded in $\C$ is called {\it isoradial} iff each face is inscribed
into a circle of a common radius $\mesh$. If all circle centers are inside the corresponding
faces, then one can naturally embed the dual graph $\G^*$ in $\C$ isoradially with the same
$\mesh$, taking the circle centers as vertices of $\G^*$. The name {\it rhombic lattice} is
due to the fact that all quadrilateral faces of the corresponding bipartite graph $\L$ (having
the vertex set $\G\cup\G^*$) are rhombi with sides of length $\mesh$ (see
Fig.~\ref{Fig:IsoGraph}A). We will often work with rhombi half-angles, denoted by $\theta$,
for which we also require the following mild but indispensable and widely used assumption
(see, e.g., \cite{Cia78}, pp. 124 and 130, where the similar assumption is called Zl\'amal's
condition):
\begin{quotation}
{\bf \sps} the rhombi half-angles are uniformly bounded from $0$ and $\frac{1}{2}\pi$ (in
other words, all these angles belong to $[\eta,\frac{1}{2}\pi-\eta]$ for some fixed $\eta>0$),
i.e., there are no ``too flat'' rhombi in $\L$.
\end{quotation}
Note that condition \sps~ implies that for each $u_1,u_2\in\G$ the Euclidean distance
$|u_2-u_1|$ and the combinatorial distance $\mesh\cdot d_\G(u_1,u_2)$ (where $d_\G(u_1,u_2)$
is the minimal number of vertices in the path connecting $u_1$ and $u_2$ in $\G$) are
comparable. Below we often use the notation {\it const} for absolute positive constants that
does not depend on the mesh $\mesh$ or the graph structure but, in principle, {\it may depend
on $\eta$}.

The function $H:\O^\mesh_\G\to\R$ defined on some subset (discrete domain) $\O^\mesh_\G $ of
$\G$ is called {\it discrete harmonic}, if
\begin{equation}
\label{DiscrHarmDef} \sum_{s=1}^n\tan\theta_s\cdot(H(u_s)\!-\!H(u))=0
\end{equation}
at all $u\in\O^\mesh_\G$ where the left-hand side makes sense.
Here $\theta_s$ denotes the half-angles of the corresponding rhombi,
see also Fig. \ref{Fig:Notations}B for notations.
As usual, this definition is
closely related to the {\it random walk on $\G$} such that the probability to make the next
step from $u$ to $u_k$ is proportional to $\tan\theta_k$.
Namely, $\mathrm{RW}(t+1)=\mathrm{RW}(t)+\xi^{(t)}_{\mathrm{RW}(t)}$, where the increments
$\xi^{(t)}$ are independent with distributions
\[
\bP(\xi_u=u_k-u)=\frac{\tan\theta_k}{\sum_{s=1}^n\tan\theta_s}\quad \mathrm{for}\ \ k=1,..,n.
\]
Under our assumption all these probabilities are uniformly bounded from $0$.
Note that the choice of $\tan\theta_s$ as the edge weights in (\ref{DiscrHarmDef}) gives
\begin{equation}
\label{RandomWalkParam} \bE[\Re\xi_u]=\bE[\Im\xi_u]=0\quad\mathrm{and}\quad
\begin{array}{l}
\displaystyle \bE[(\Re\xi_u)^2]=\bE[(\Im\xi_u)^2]=T_u, \vphantom{\big|_|}\cr
\bE[\Re\x_u\Im\x_u]=0\vphantom{\big|^|},
\end{array}
\end{equation}
where $T_u=\mesh^2\cdot {\sum_{s=1}^n\sin2\theta_s}\big/{\sum_{s=1}^n\tan\theta_s}$ (see Lemma
\ref{ApproxLemma}). Our results may be directly interpreted as the convergence of the hitting
probabilities for this random walk. Moreover, condition \sps~ implies that quadratic
variations satisfy {$0<\const\cdot\mesh^2\le T_u\le 2\mesh^2$}, and so one can define a proper
lazy random walk (or make a time re-parametrization) according to (\ref{RandomWalkParam}) so
that it converges to standard 2D Brownian motion.

\subsection{Main results}
Let $\O^\mesh_\G\ss\G$ be some bounded, simply connected discrete domain and
$\Int\O^\mesh_\G$, $\pa\O^\mesh_\G$ denote the sets of interior and boundary vertices,
respectively (see Sect. \ref{SectBasicDefs} for more accurate definitions). For
$u\in\Int\O^\mesh_\G$ and $E\ss\pa\O^\mesh_\G$ the discrete harmonic measure $\dhm\mesh u E
{\O^\mesh_\G}$ is the probability of the event that the random walk on $\G$ starting at $u$
first exits $\O^\mesh_\G$ through $E$. Equivalently, $\dhm\mesh\cdot E {\O^\mesh_\G}$ is the
unique solution of the following discrete Dirichlet boundary value problem:
\begin{itemize}
\item $\dhm\mesh{\ccdot} E {\O^\mesh_\G}$ is discrete harmonic everywhere in $\O^\mesh_\G$;

\item $\dhm\mesh a E {\O^\mesh_\G}=1$ for $a\in E$ and
$\dhm\mesh a E {\O^\mesh_\G}=0$ for $a\in\pa\O^\mesh_\G\setminus E$.
\end{itemize}

We prove {\it uniform} (with respect to the shape $\O^\mesh_\G$ and the structure of the
underlying isoradial graph) {\it convergence} of the basic objects of the discrete potential
theory and their discrete {\it gradients} (which are discrete holomorphic functions defined on
subsets of $\DS=\L^*$, see Sect.~\ref{SectDiscrHol} and Definition~\ref{DefUniformC1conv} for
further details) to continuous counterparts. Namely, we consider
\begin{itemize}
\item
solution of the discrete Dirichlet problem with continuous boundary values;
\item
discrete harmonic measure $\dhm\mesh{\ccdot}{a^\mesh b^\mesh}{\O^\mesh_\G}$ of boundary arcs
$a^\mesh b^\mesh\ss\pa\O^\mesh_\G$;
\item
discrete Green's function $G^\mesh_{\O^\mesh_\G}(\ccdot;v^\mesh)$,
$v^\mesh\in\Int\O^\mesh_\G$;
\item
discrete Poisson kernel
\begin{equation}
\label{DefDiscrPoisson} P^\mesh(\ccdot;v^\mesh;a^\mesh;\O^\mesh_\G):=
\frac{\dhm\mesh{\ccdot}{\{a^\mesh\}}{\O^\mesh_\G}}{\dhm\mesh{v^\mesh}{\{a^\mesh\}}{\O^\mesh_\G}},\quad
a^\mesh\in\pa\O^\mesh_\G,
\end{equation}
normalized at the interior point $v^\mesh\in\Int\O^\mesh_\G$;
\item
discrete Poisson kernel $P^\mesh_{o^\mesh}(\ccdot;a^\mesh;\O^\mesh_\G)$,
$a^\mesh\in\pa\O^\mesh_\G$, normalized at the boundary point $o^\mesh\in\pa\O^\mesh_\G$ by
some analogue of the condition $[\pa_n P](o^\mesh)=1$ (we assume that the boundary $\pa
\O^\mesh_\G$ is ``straight'' near $o^\mesh$, see precise definitions in
Sect.~\ref{SectBoundaryNorm}).
\end{itemize}

\subsection{Organization of the paper}
We begin with the exposition of  basic facts concerning discrete harmonic and discrete
holomorphic functions on isoradial graphs. The larger part of Sect.~\ref{SectBasicFacts}
follows \cite{Duf68}, \cite{Mer01}, \cite{Ken02}, \cite{Mer07} and \cite{Buck08}.
Unfortunately, none of these papers contains all the preliminaries that we need. Besides, the
basic notation (sign and normalization of the Laplacian, definition of the discrete
exponentials and so on) varies from source to source, so for the convenience of the reader we
collected all preliminaries in the same place. Note that our notation (e.g., the normalization
of discrete Green's functions and the parametrization of discrete exponentials) is chosen to
be as close in the limit to the standard continuous objects as possible. Also, we prefer to
deal with functions rather than to use the language of forms or cochains \cite{Mer07} which is
more adapted for the topologically nontrivial cases.

The main part of our paper is Sect.~\ref{SectConvThms}, where the convergence theorems are
proved. The proofs essentially use compactness arguments, so it does not give any 
estimate for the convergence rate. Thus, as in \cite{Sm07}, we derive the ``uniform''\
convergence from the ``pointwise''\ one, using the compactness of the set of bounded
simply connected domains in the Carath\'edory topology (see Proposition~\ref{PropUniC1}).
The other ingredients are the classical Arzel\`a-Ascoli theorem, which allows us to choose a
convergent subsequence of discrete harmonic functions (see Proposition~\ref{PropUniformComp})
and the weak Beurling-type estimate (Proposition~\ref{PropWeakBeurling}) which we use in order
to identify the boundary values of the limiting harmonic function. We prove
\mbox{$C^1$-convergence}, but stop short of discussing the $C^\infty$ topology since there is
no straightforward definition of the second discrete derivative for functions on isoradial
graphs (see Sect.~\ref{SectElemFactsHol}). Note however that a way to overcome this difficulty
was suggested in \cite{Buck08}.

\medskip

\noindent {\bf Acknowledgments.} We would like to thank Vincent Beffara for many helpful
comments. Some parts of this paper were written at the MFO, Oberwolfach (during the
Oberwolfach-Leibniz Fellowship of the first author), and at the IH\'ES, Bures-sur-Yvette. The
authors are grateful to the Institutes for the hospitality.

This research was supported by the Swiss N.S.F., by the European Research Council AG CONFRA,
by EU RTN CODY and, on the final stage, by the Chebyshev Laboratory (Department of Mathematics
and Mechanics, Saint-Petersburg State University) under the grant 11.G34.31.0026 of the
Government of the Russian Federation. The first author was partly funded by P.Deligne's 2004
Balzan prize in Mathematics and by the grant MK-7656.2010.1.

\section{Discrete harmonic and holomorphic functions. Basic facts}
\setcounter{equation}{0}\label{SectBasicFacts}

\subsection{Basic definitions. Approximation property.}
\label{SectBasicDefs} Let $\G=\G^\mesh$ be some {\it infinite} isoradial graph embedded
into~$\C$ and $V_{\O^\mesh}\ss\G$ be some {\it connected} subset of vertices (identified with
points in $\C$). Let $E_{\O^\mesh}$ be the set of all edges (open intervals in $\C$) incident
to $V_{\O^\mesh}$ and $F_{\O^\mesh}$ be the set of all faces (open polygons in $\C$) incident
to $E_{\O^\mesh}$.

We call $\bm{\O^\mesh}:=F_{\O^\mesh}\cup E_{\O^\mesh}\cup V_{\O^\mesh}\ss\C$ the {\bf
polygonal representation} of a {\bf discrete domain} $\bm{\O^\mesh_\G}:=\Int\O^\mesh_\G\cup
\pa\O^\mesh_\G$, where {\it interior} and {\it boundary vertices} are defined as
\[
\Int\O^\mesh_\G:=V_{\O^\mesh}\quad \mathrm{and}\quad
\pa\O^\mesh_\G:=\{(a\,;(a_{\mathrm{int}}a)):\ a_{\mathrm{int}}\in V_{\O^\mesh},\
(a_{\mathrm{int}}a)\in E_{\O^\mesh},a\notin V_{\O^\mesh}\},
\]
respectively. Further, we say that $\O^\mesh_\G$ is simply connected, if $\O^\mesh$ is simply
connected. The reason for this definition of $\pa\O^\mesh_\G$ is that the same $a$ may serve
as several different boundary vertices, if it can be approached from $\Int\O^\mesh_\G$ by
several edges -- see e.g. vertices $b$ and $c$ in the Fig.~\ref{Fig:DiscrDomain}A). However,
when no confusion arises, we will often treat $\pa\O^\mesh_\G$ as a subset of $\G$, not
indicating explicitly the corresponding outgoing edges.

Below we often need some natural {\it discretizations of standard continuous domains} (e.g.,
discs and rectangles). For an open convex $D\ss\C$ we introduce $D^\mesh_\G\ss\G$ and its
polygonal representation $D^\mesh\ss\C$ by defining $\Int\O^\mesh_\G=V_{D^\mesh}$ as the
vertices of the (largest) connected component of $\G$ lying inside $D$ (see
Fig.~\ref{Fig:DiscrHalfPlane}B, Fig.~\ref{Fig:DiscrDisc}A).

\begin{figure}
\centering{
\begin{minipage}[c]{0.44\textwidth}
\centering{\includegraphics[height=0.35\textheight]{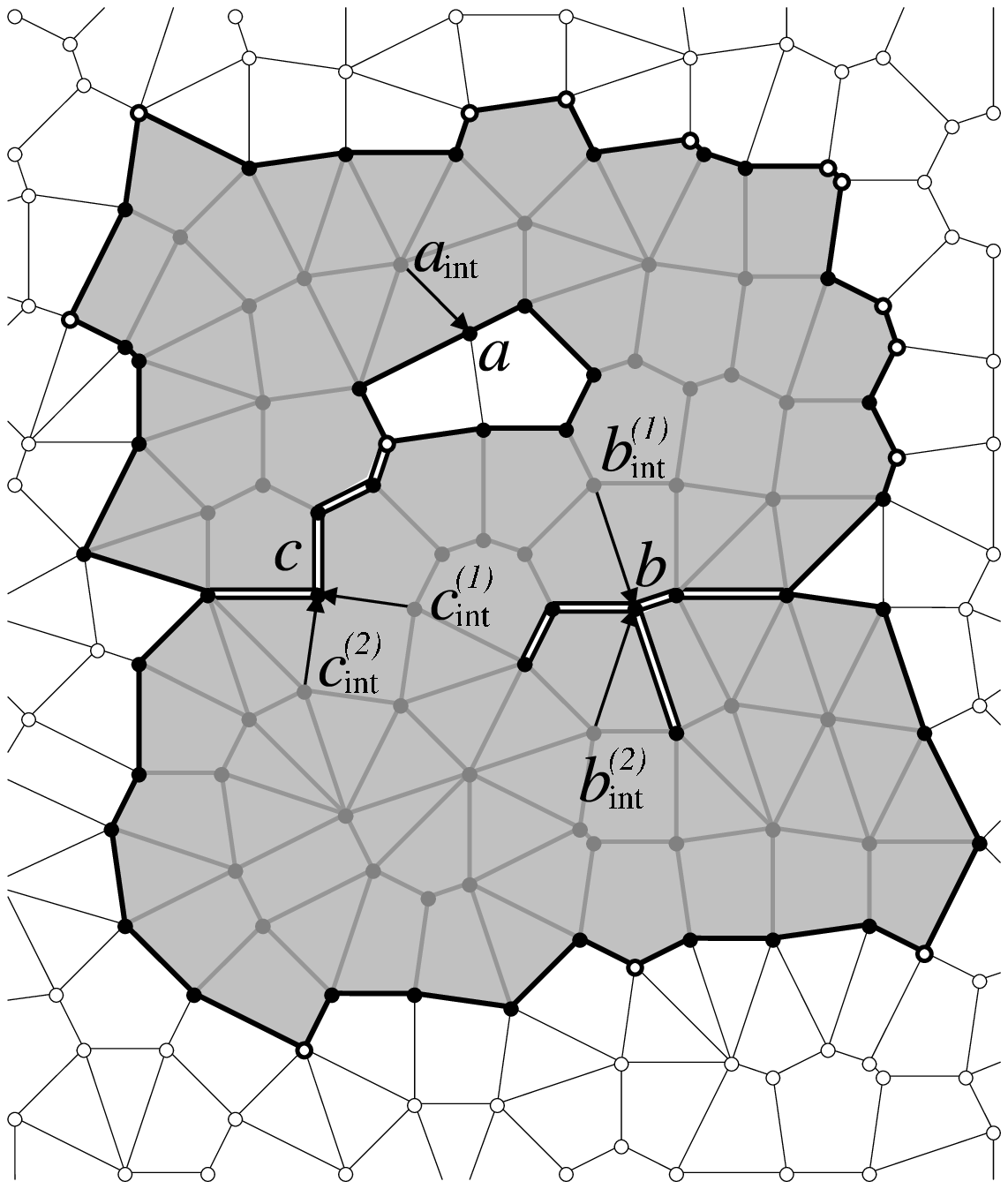}}

\centering{\textsc{(A)}}
\end{minipage}
\hskip 0.05\textwidth
\begin{minipage}[c]{0.5\textwidth}
\centering{\includegraphics[height=0.35\textheight]{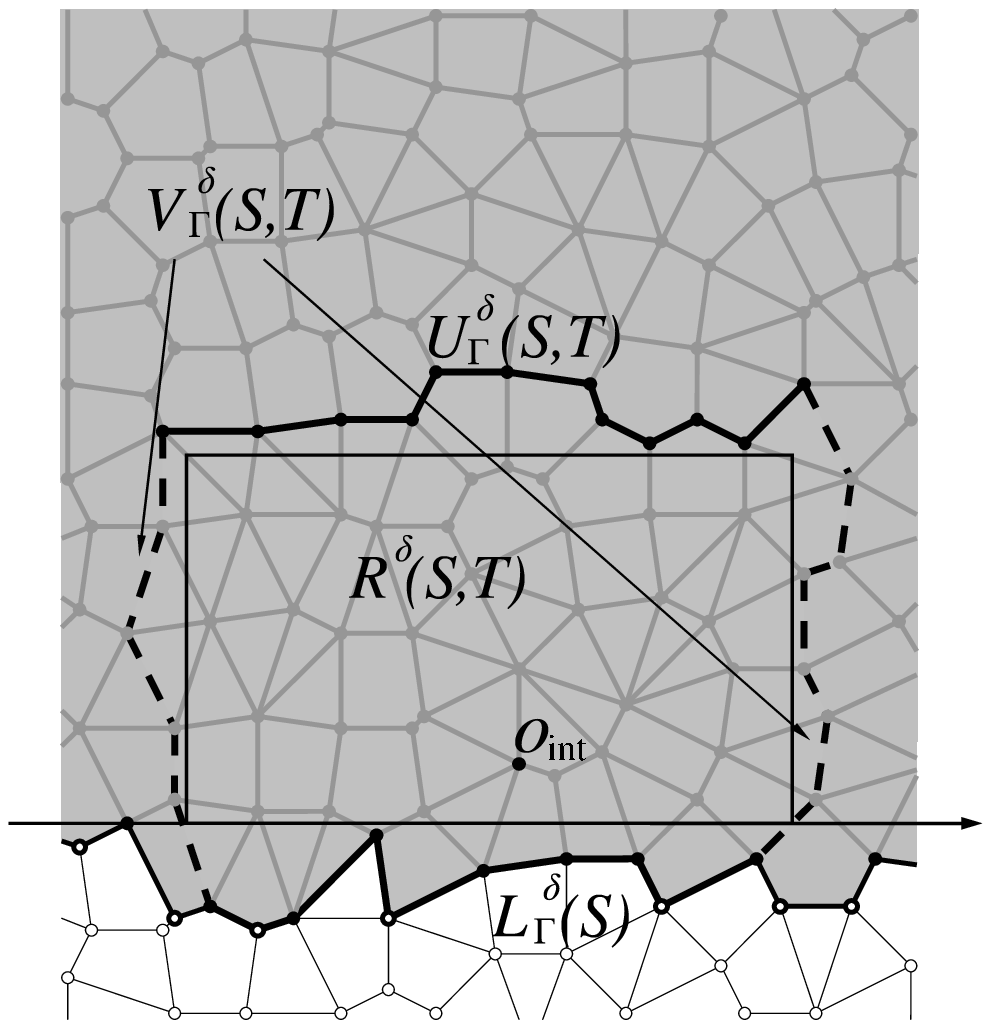}}

\centering{\textsc{(B)}}
\end{minipage}
} \caption{\label{Fig:DiscrDomain}\textsc{(A)} Discrete domain. The interior vertices are
gray, the boundary vertices are black and the outer vertices are white. Both $b$ and $c$ have
two interior neighbors, and so we treat, e.g., $(b\,;(b_{\mathrm{int}}^{(1)}b))$ and
$(b\,;(b_{\mathrm{int}}^{(2)}b))$ as different elements of $\pa\O^\mesh_\G$.
\label{Fig:DiscrHalfPlane}\textsc{(B)} Discrete half-plane $\H^\mesh$ and discrete rectangle
$R^\mesh(S,T)$. The lower, upper and vertical parts of $\pa R^\mesh_\G(S,T)$ are denoted by
$L^\mesh_\G(S)$, $U^\mesh_\G(S,T)$ and $V^\mesh_\G(S,T)$, respectively.}
\end{figure}

Let
\begin{equation}
\label{WeightGDef}
\weightG{u}:=\frac{\mesh^2}{2}\sum_{u_s\sim u}\sin 2\theta_s,
\end{equation}
be the {\it weight} of a vertex $u\in\G$, where $\theta_s$ are the half-angles of the
corresponding rhombi. Note that $\weightG{u}$ is the area of a dual face $W(u)=w_1w_2..w_n$
(see Fig. \ref{Fig:Notations}B).

Let $\phi:\O^\mesh\to\C$ be a Lipschitz (i.e., satisfying $|\phi(u_1)-\phi(u_2)|\le
C|u_1-u_2|$) function and \mbox{$\phi^\mesh:=\phi|_{\O^\mesh_\G}$} be its restriction to
$\O^\mesh_\G$. Note that all points in a dual face $W(u)$ are \mbox{$\mesh$-close} to its
center $u$. Thus, approximating values of $\phi$ on $W(u)$ by $\phi(u)$ and taking into
account that
$\Area(\O^\mesh\setminus\bigcup_{u\in\Int\O^\mesh_\G}W(u))\le\mesh\cdot\Length(\pa\O^\mesh)$,
we arrive at the simple inequality
\begin{equation}
\label{DiscrIntApprox} \lt|\sum_{u\in\Int\O^\mesh_\G}\phi^\mesh(u)\weightG{u}
-\iint_{\O^\mesh}\phi(x\!+\!iy)dxdy\rt|\le C\mesh\cdot \Area(\O^\mesh) + M\delta\cdot
\Length(\pa\O^\mesh)
\end{equation}
with the same constant $C$ and $M:=\sup\{|\phi(z)|, z\in\O^\mesh:
\dist(z,\pa\O^\mesh)\le\mesh\}$.

\begin{definition}
Let $\O^\mesh_\G$ be some connected discrete domain and $H:\O^\mesh_\G\to \R$. We define the
{\bf discrete Laplacian} of $H$ at $u\in\Int \O^\mesh_\G$ by
\[
[\D^\mesh H](u):=\frac{1}{\weightG{u}}\sum_{u_s\sim u} \tan\theta_s\cdot[H(u_s)\!-\!H(u)]
\]
(see Fig.~\ref{Fig:Notations}B for notations). We call $H$ {\bf discrete harmonic} in
$\O^\mesh_\G$ iff $[\D^\mesh H](u)=0$ at all interior vertices $u\in\Int \O^\mesh_\G$.
\end{definition}

It is easy to see that discrete harmonic functions satisfy the {\bf maximum principle}:
\begin{equation}
\label{MaxPrinciple} \max_{u\in\O^\mesh_\G} H(u) = \max_{a\in\pa\O^\mesh_\G} H(a).
\end{equation}

Further, a simple calculation shows that the {\bf discrete Green's formula}
\begin{equation}
\label{GreenFormula} \sum_{u\in\Int\O^\mesh_\G} [H\D^\mesh G -  G\D^\mesh H](u)\weightG{u}\
=\!\! \sum_{a\in \pa\O^\mesh_\G}\tan\theta_{a_{\mathrm{int}}a}\cdot
[H(a_{\mathrm{int}})G(a)-H(a)G(a_{\mathrm{int}})]
\end{equation}
holds true for any two functions $H,G:\O^\mesh_\G\to\R$. Here and below, for a boundary vertex
$(a\,;(a_{\mathrm{int}}a))$, $\theta_{a_{\mathrm{int}}a}$ denotes the half-angle of the
rhombus having ${a_{\mathrm{int}}a}$ as a diagonal.

\begin{lemma}[\bf approximation property]
\label{ApproxLemma} Let $\phi\in C^3$ be a smooth function  defined in the disc
$B(u,2\mesh)\ss\C$ for some $u\in\G$. Denote by $\phi^\mesh$  its restriction to $\G$. Then

\smallskip

\noindent (i)\phantom{i} $\D^\mesh \phi^\mesh \equiv 0$, if $\phi$ is constant or a linear
function, and

\noindent \phantom{(ii)} $\D^\mesh \phi^\mesh \equiv \D\phi \equiv 2(a+c)$, if
$\phi(x+iy)\equiv ax^2+bxy+cy^2$ is quadratic in $x$ and $y$.

\smallskip

\noindent (ii)
\[
\left|\,[\D^\mesh \phi^\mesh](u) - [\D\phi](u)\,\right| \le  \const\cdot\mesh\cdot
\sup_{B(u,2\mesh)} |D^3 \phi|.
\]
\end{lemma}

\begin{proof} We start by enumerating neighbors of $u$ as
$u_1,\dots,u_n$ and its neighbors on the dual lattice as $w_1,\dots,w_n$
-- see Fig.~\ref{Fig:Notations}B).
 Obviously, $\D^\mesh \phi^\mesh \equiv 0$, if $\phi$ is a constant. Since
\[
\sum_{u_s\sim u}\tan\theta_s \cdot (u_s\!-\!u) = -i\sum_{u_s\sim u}(w_{s+1}\!-\!w_s)=0,
\]
one obtains $\D^\mesh \phi^\mesh \equiv 0$ for linear functions $x=\Re u$ and $y=\Im u$.
Similarly,
\[
\sum_{u_s\sim u}\tan\theta_s \cdot (u_s^2\!-\!u^2) = -i\sum_{u_s\sim
u}(w_{s+1}\!-\!w_s)(u\!+\!u_s)=-i\sum_{u_s\sim u}(w_{s+1}^2\!-\!w_s^2)=0,
\]
so $\D^\mesh \phi^\mesh \equiv 0$ for $x^2\!-\!y^2=\Re u^2$ and $2xy=\Im u^2$. The result for
$x^2\!+\!y^2$ follows from
\[
\sum_{u_s\sim u}\tan\theta_s \cdot |u_s\!-\!u|^2 = 2\mesh^2\sum_{u_s\sim u} \sin 2\theta_s =
4\weightG{u},
\]
thus proving (i). Finally, Taylor formula implies (ii).
\end{proof}

\subsection{Green's function. Dirichlet problem. Harnack lemma. Lipschitzness}

\begin{definition} Let $u_0\in\G$. We call $H=G_\G(\ccdot;u_0):\G\to\R$ the {\bf free Green's
function} iff it satisfies the following:

\begin{quotation}
\noindent (i)\phantom{ii} $[\D^\mesh H](u)=0$ for all $u\ne u_0$ and $[\D^\mesh H](u_0)\cdot
\weightG{u_0} =1$;

\smallskip

\noindent (ii)\phantom{i} $H(u)=o(|u\!-\!u_0|)$ as $|u-u_0|\to\iy$;

\smallskip

\noindent (iii) $ H(u_0)= \frac{1}{2\pi}(\log\mesh \!-\! \g_{\mathrm{Euler}}\!-\!\log 2)$,
where $\g_{\mathrm{Euler}}$ is the Euler constant.
\end{quotation}
\end{definition}

\begin{remark} We use a nonstandard normalization at $u_0$ (usually the additive constant
is chosen so that $G(u_0;u_0)=0$) in order to have convergence to the standard continuous
Green's function $\frac{1}{2\pi}\log|u\!-\!u_0|$ as the mesh $\mesh$ goes to zero.
\end{remark}

\begin{theorem}[\bf Kenyon]
\label{ThmKenyonHarm} There exists a unique Green's function $G_\G(\ccdot;u_0)$. Moreover, it
satisfies
\begin{equation}
\label{Gasympt} G_\G(u;u_0)= \frac{1}{2\pi}\log |u\!-\!u_0| +
O\lt(\frac{\d^2}{|u\!-\!u_0|^2}\rt),\quad u\ne u_0,
\end{equation}
uniformly with respect to the shape of the isoradial graph $\G$ and $u_0\in\G$.
\end{theorem}

\begin{proof}
This asymptotic form for isoradial graphs was first obtained in \cite{Ken02}. Some small
improvements (the correct additive constant and the order of the remainder) were done in
\cite{Buck08}. We give a sketch of Kenyon's beautiful proof in
Appendix~\ref{SectAKenyonFreeGreen}.
\end{proof}

Let $\O^\mesh_\G$ be some {\it bounded} connected discrete domain. It is well known that for
each $f:\pa\O^\mesh_\G\to\R$ there exists a unique discrete harmonic function $H$ in
$\O^\mesh_\G$ such that $H|_{\pa\O^\mesh_\G}=f$ (e.g., $H$ minimizes the corresponding
Dirichlet energy, see \cite{Duf68}). Clearly, $H$ depends on $f$ linearly, and so
\[
H(u)=\sum_{a\in\pa\O^\d_\G} \dhm{\mesh}{u}{\{a\}}{\O^\mesh_\G}\cdot f(a)
\]
for all $u\in \O^\mesh_\G$, where $\dhm{\mesh}{u}{\ccdot}{\O^\mesh_\G}$ is some probabilistic
measure on $\pa\O^\mesh_\G$ which is called {\bf harmonic measure} at $u$. It is harmonic as a
function of $u$ and has a standard interpretation as the exit probability for the random walk
on $\G$ (the measure of a set $E\subset\pa\O^\mesh_\G$ is the probability that the random walk
started from $u$ exits $\O^\mesh_\G$ through $E$).

\begin{definition}
For $u_0\in\Int\O^\mesh_\G$, we call $H=G_{\Omega^\mesh_\G}(\ccdot;u_0)$ the {\bf Green's
function in~$\bm{\O^\mesh_\G}$} iff

(i)\phantom{i} $[\D^\mesh H](u)=0$ for all interior vertices $u\in \Int\O^\mesh_\G$ except
$u_0$ and

$\phantom{(ii)}\ [\D^\mesh H](u_0)\cdot\weightG{u_0}=1$;

(ii) $H=0$ on the boundary $\pa\O^\mesh_\G$.
\end{definition}

\noindent Note that these properties determine $G_{\Omega^\mesh_\G}(\ccdot;u_0)$ uniquely.
Namely, $G_{\Omega^\mesh_\G}^{\ }= G_\G-G_{\Omega^\mesh_\G}^*$, where
\[
G_{\Omega^\mesh_\G}^*=G_{\Omega^\mesh_\G}^*(\ccdot;u_0
):=\sum_{a\in\pa\O^\mesh_\G}\dhm{\mesh}{\ccdot}{\{a\}}{\O^\mesh_\G}\cdot G_\G(a;u_0)
\]
is a unique solution of the discrete boundary value problem
\[
\D^\mesh G_{\Omega^\mesh_\G}^*=0\ \ \mathrm{in}\ \ \O^\mesh_\G,\qquad
G_{\Omega^\mesh_\G}^*=G_\G(\ccdot;u_0)\ \mathrm{on}\ \pa\O^\mesh_\G.
\]
Applying Green's formula (\ref{GreenFormula}) to $H=\dhm{\mesh}{\ccdot}{\{a\}}{\O^\mesh_\G}$
and $G=G_{\Omega^\mesh_\G}(\ccdot;u_0)$, one obtains
\begin{equation}
\label{oAsG} \dhm{\mesh}{u_0}{\{a\}}{\O^\mesh_\G} = -\tan\theta_{a_{\mathrm{int}}a} \cdot
G_{\Omega^\mesh_\G}(a_{\mathrm{int}};u_0),\quad \mathrm{where}\ \
a=(a\,;(a_{\mathrm{int}}a))\in\pa\O^\mesh_\G.
\end{equation}

It was noted by U.~B\"ucking \cite{Buck08} that, since the remainder in (\ref{Gasympt}) is of
order $O(\mesh^2|u\!-\!u_0|^{-2})$, one can directly use R.~Duffin's ideas \cite{Duf53} in
order to derive the Harnack Lemma for discrete harmonic functions.

Recall that $B^\mesh_\G(z,r)\ss\G$ denotes the discretization of an open disc $B(z,r)\ss\C$.

\begin{proposition}[\bf discrete Harnack Lemma]
\label{PropHarnack} Let $u_0\in\G$ and $H:B^\mesh_\G(u_0,R)\to\R$ be a nonnegative discrete
harmonic function.\\
(i)\phantom{i} If $u_1\sim u_0$, then
\[
|H(u_1)-H(u_0)|\le \const\cdot\frac{\mesh H(u_0)}{R}.
\]
(ii) If $u_1,u_2\in B^\mesh_\G(u_0,r)\ss\Int B^\mesh_\G(u_0,R)$, then
\[
\exp\lt[-\const\cdot\frac{r}{R-r}\rt] \le \frac{H(u_2)}{H(u_1)}\le
\exp\lt[\const\cdot\frac{r}{R-r}\rt].
\]
\end{proposition}

\begin{remark}
In Sect.~\ref{SectBoundaryNorm} we also give a version of the boundary Harnack principle which
compares the values of a positive harmonic function in the bulk with its normal derivative on
a ``straight'' part of the boundary (see Proposition~\ref{PropBoundHar}).
\end{remark}

\begin{proof}
In order to make our presentation complete, we recall briefly the arguments from \cite{Duf53}
and \cite{Buck08} in Appendix~\ref{SectAHarnack}.
\end{proof}

\begin{corollary}[\bf Lipschitzness of discrete harmonic functions]
\label{CorLipHarm} Let $H$ be discrete harmonic in $B^\mesh_\G(u_0,R)$ and $u_1,u_2\in
B^\mesh_\G(u_0,r)\ss\Int B^\mesh_\G(u_0,R)$. Then
\[
|H(u_2)-H(u_1)|\le \const\cdot\frac{M|u_2\!-\!u_1|}{R-r},\quad \mathit{where}\quad
M=\max_{B^\mesh_\G(u_0,R)}|H(u)|.
\]
\end{corollary}
\begin{proof}
By assumption \sps~we can find a path $u_1=v_0v_1v_2...v_{k-1}v_k=u_2$,  connecting $u_1$ and
$u_2$ inside $B^\mesh_\G(u_0,r)$, such that $k\le \const\cdot \mesh^{-1}|u_2\!-\!u_1|$. Since
$0\le H\!+\!M\le 2M$, applying Harnack's inequality to $H+M$, one gets
\[
|H(u_2)-H(u_1)|\le \sum_{j=0}^{k-1}|H(v_{j+1})-H(v_j)| \le
\const\cdot\frac{|u_2\!-\!u_1|}{\mesh}\cdot\frac{\mesh M}{R-r}. \qedhere
\]
\end{proof}

\subsection{Weak Beurling-type estimates}

The following simple fact is based on the approximation property (Lemma~\ref{ApproxLemma}) for
the discrete Laplacian on isoradial graphs.
\begin{lemma}
\label{DiscExit} Let $u_0\in\G$, $r\!>\!0$ and $B^\mesh_\G(u_0,r)$ be the discretization of a
disc $B(u_0,r)$ (see Fig. \ref{Fig:DiscrDisc}A). Let $a,b\in\pa B^\mesh_\G(u_0,r)$ be two
boundary vertices such that
\[
\textstyle \arg (b\!-\!u_0)-\arg (a\!-\!u_0) \ge \frac{1}{4}\pi.
\]
Then,
\[
\textstyle \dhm{\mesh}{u}{ab}{B^\mesh_\G(0,r)}\ge \const > 0\quad \mathit{for\ all}\ u\in
B^\mesh_\G(u_0,\frac{1}{2}r),
\]
where $a b$ denotes the discrete counter clockwise arc from $a$ to $b$.
\end{lemma}

\begin{proof}
Fix some small $\rho>0$ and a smooth function $\phi_0:B(0,1\!+\!\rho)\to\R$ such that

(ia) $\phi_0(z)\le 1$ for all $z=re^{i\phi}$, $r\in(1\!-\!\rho,1\!+\!\rho)$,
$\phi\in[0,\frac{1}{4}\pi]$;

(ib) $\phi_0(z)\le 0$ for all $z=re^{i\phi}$, $r\in(1\!-\!\rho,1\!+\!\rho)$,
$\phi\in[\frac{1}{4}\pi,2\pi]$;

(ii)\phantom{i} $\phi_0$ is subharmonic, moreover $[\D\phi_0](\z)\ge \const
>0$ everywhere in $B(0,1\!+\!\rho)$;

(iii) $\phi_0(z)\ge \const >0$ for all $z\in B(0,\frac{1}{2}\!+\!\rho)$.


\noindent For instance, one can take $\phi_0(z):=h(z)-c+d|z|^2$, where $h$ is the (continuous)
harmonic measure of the arc
$\{\z:|\z|=1\!+\!\rho:\arg\z\in[\frac{1}{12}\pi,\frac{1}{6}\pi]\}$; $c>0$ is chosen so that
(ib) and (iii) are fulfilled ($c$ exists, if $\rho$ is small enough); and $d>0$ is
sufficiently small.

Let
\[
\phi^\mesh(u):=\phi_0\lt(\frac{u-u_0}{a-u_0}\rt)\quad\mathrm{for}\quad u\in B^\mesh_\G(u_0,r).
\]
Then, $\phi^\mesh\le 1$ on the discrete arc $a b$ and $\phi^\mesh\le 0$ on the complementary
arc $b a$.

If $\mesh/r$ is small enough, then, due to (ii) and Lemma~\ref{ApproxLemma} (approximation
property), $\phi^\mesh$ is discrete subharmonic in $B^\mesh_\G(u_0,r)$. Using the maximum
principle, one obtains
\[
\textstyle \dhm{\mesh}{u}{a b}{B^\mesh_\G(0,r)}\ge \phi^\mesh(u)\ge\const>0\quad \mathrm{for\
all}\ u\in B^\mesh_\G(0,\frac{1}{2}r).
\]
If $\mesh/r\ge \const>0$, then the claim is trivial, since the random walk starting at $u_0$
can reach the discrete arc $a b$ in a uniformly bounded number of steps.
\end{proof}

Let $\O^\mesh_\G$ be some connected discrete domain, $u\in\O^\mesh_\G$ and
$E\ss\pa\O^\mesh_\G$. We set
\[
\dist_{\O^\mesh_\G}(u;E) := \inf\{R:\ u\ \mathrm{and}\ E\ \mathrm{are\ connected\ in}\
\O^\mesh_\G\cap B(u,R)\}.
\]
The following Proposition is a simple discrete version of the classical Beurling estimate with
a (sharp) exponent $1/2$ replaced by some (small) positive $\beta$.
\begin{proposition}[\bf weak Beurling-type estimates]
\label{PropWeakBeurling} There exists an absolute constant $\b>0$ such that for any simply
connected discrete domain $\O^\mesh_\G$, interior vertex $u\in\Int\O^\mesh_\G$ and some part
of the boundary $E\subset\pa\O^\mesh_\G$ one has
\[
\dhm{\mesh}{u}{E}{\O^\mesh_\G}\le \const\cdot
\lt[\frac{\dist(u;\pa\O^\mesh_\G)}{\dist_{\O^\mesh_\G}(u;E)}\rt]^\b \quad \mathit{and} \quad
\dhm{\mesh}{u}{E}{\O^\mesh_\G}\le \const\cdot \lt[\frac{\diam
E}{\dist_{\O^\mesh_\G}(u;E)}\rt]^\b.
\]
Above we set $\diam E:=\mesh$, if $E$ consists of a single vertex.
\end{proposition}

{
\begin{figure}
\centering{
\begin{minipage}[c]{0.44\textwidth}
\centering{\includegraphics[height=0.35\textheight]{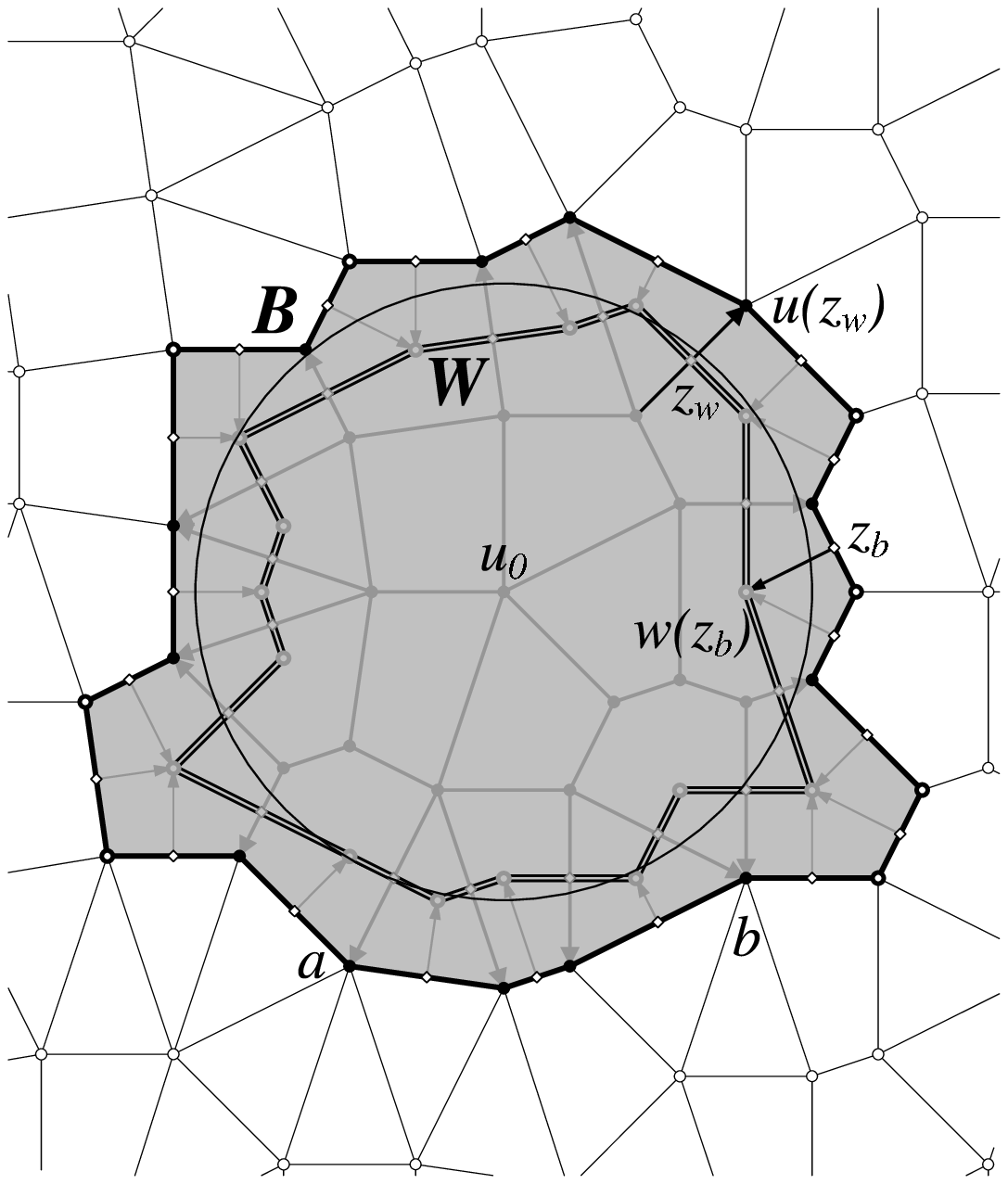}}
\centering{\textsc{(A)}}
\end{minipage}
\hskip -0.05\textwidth
\begin{minipage}[c]{0.55\textwidth}

\vspace{0.025\textheight}

\centering{\includegraphics[height=0.3\textheight]{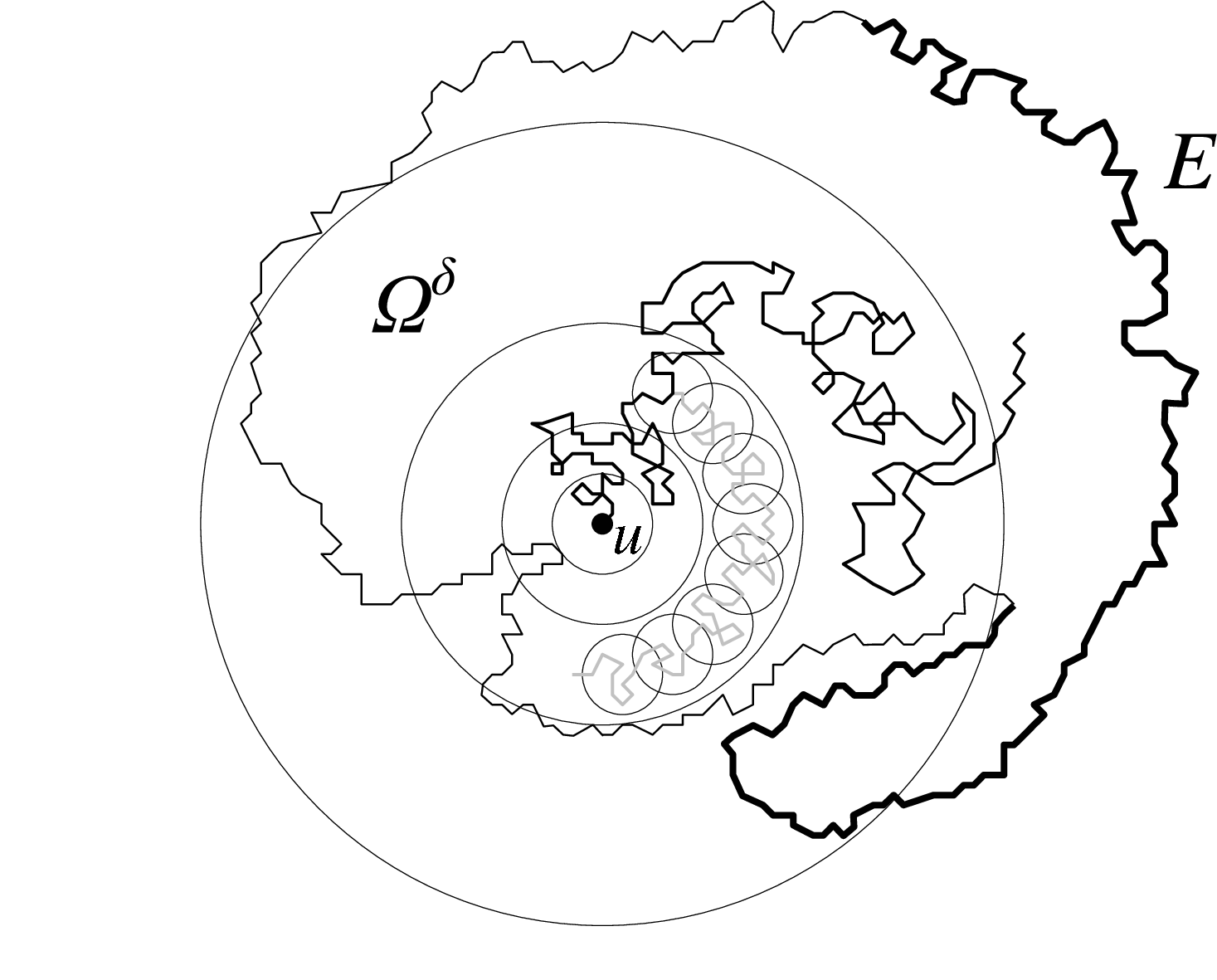}}

\vspace{0.025\textheight}

\centering{\textsc{(B)}}
\end{minipage}
} \caption{\label{Fig:DiscrDisc}\textsc{(A)} A discrete disc. The ``black''\ polygonal
boundary $B$ and the ``white''\ contour $W$ are shown together with the correspondences
$z\mapsto u(z)$, $z\in W_\DS$, and $z\mapsto w(z)$, $z\in B_\DS$.
\label{Fig:Beurling}\textsc{(B)} The proof of the weak Beurling-type estimate (Proposition
\ref{PropWeakBeurling}). The probability that the random walk makes a whole turn inside the
annulus (and so hits the boundary $\pa\O^\mesh$) is uniformly bounded from $0$ due to Lemma
\ref{DiscExit}.}
\end{figure}

}
\begin{proof}
The proof is quite standard. Let $d=\dist(u;\pa\O^\mesh_\G)$ and $r=\dist_{\O^\mesh_\G}(u;E)$.
Recall that $\dhm\mesh{u}{E}{\O^\mesh_\G}$ is equal to the probability that the random walk
starting at $u$ first hits the boundary of $\O^\mesh_\G$ \mbox{inside $E$}. Using
Lemma~\ref{DiscExit} (see Fig.~\ref{Fig:Beurling}B), it is easy to show that for each $d\le
r'\le\frac{1}{2}r$ the probability to cross the annulus $B(u,2r')\setminus B(u,r')$ inside
$\O^\mesh_\G$ without touching its boundary is bounded above by some absolute constant $p<1$
that does not depend on $r'$ and the shape of $\O^\mesh_\G$. Hence,
\[
\dhm\mesh{u}{E}{\O^\mesh_\G} \le p^{\log_2(r/d)-1} = p^{-1}\cdot (d/r)^{-\log_2 p},
\]
so the first estimate holds true with the exponent $\b=-\log_2 p >0$.

To prove the second estimate, let us fix any vertex $e\in E$. By definition of
$d=\dist_{\O^\mesh_\G}(u;E)$, it's clear that $E$ and $u_0$ are disconnected in
$\O^\mesh_\G\cap B(e,\frac{1}{2}d)$ (otherwise $u_0$ and $E$ would be for sure connected in
$\O^\mesh_\G\cap B(u_0,d)$). Now one can mimic the arguments given above for annuli
$B(e,2r')\setminus B(e,r')$ with $\diam E\le r'\le\frac{1}{4}d$.
\end{proof}

\subsection{Discrete holomorphic functions. Definitions.}

\label{SectDiscrHol}

Above we discussed the theory of discrete harmonic functions defined on the isoradial graph
$\G$ (or, in a similar manner, on its dual $\G^*$). Now, following \cite{Duf53}, \cite{Mer01}
and \cite{Ken02}, we introduce the notion of {\it discrete holomorphic} functions. These are
defined either on vertices $\L=\G\cup\G^*$ of the rhombic lattice, or on the set $\DS=\L^*$ of
the rhombi centers. Note that, in contrast to similar $\G$ and $\G^*$, $\L$ and $\DS$ have
essentially different combinatorial properties, so we obtain two essentially different
definitions. As it will be shown in Sect.~\ref{SectElemFactsHol}, the first class (holomorphic
functions defined on $\L$) can be thought as couples of harmonic functions and their harmonic
conjugates, while the second (holomorphic functions defined on $\DS$) consists of gradients of
harmonic functions. We are mostly interested in the second class, but start with some
preliminaries concerning functions defined on $\L$.

\begin{definition}
\label{DefHolL} Let $z\in\DS$ be a center of the rhombus $u^{-}w^{-}u^{+}w^{+}$, where
$u^{\pm}\in\G$ and $w^{\pm}\in\G^*$ are listed in counter clockwise order. Let a function $H$
be defined on some part of $\L$ including $u^{\pm}$, $w^{\pm}$.  We define its discrete
derivatives $\dpa H$, $\dopa H$ at $z$ as
\[
[\dpa H](z):= \frac{1}{2}\lt[ \frac{H(u^{+})\!-\!H(u^{-})}{{u}^{+}\!-\!{u}^{-}} +
\frac{H(w^{+})\!-\!H(w^{-})}{{w}^{+}\!-\!{w}^{-}} \rt],
\]
\[
[\dopa H](z):=\frac{1}{2}\lt[ \frac{H(u^{+})\!-\!H(u^{-})}{\ol{u^{+}\!-\!u^{-}}} +
\frac{H(w^{+})\!-\!H(w^{-})}{\ol{w^{+}\!-\!w^{-}}} \rt].
\]
We use the same notations, if $H$ is defined on $\G$ (or $\G^*$) only, formally setting
$H|_{\G^*}:=0$ (or $H|_\G:=0$, respectively). We call $H$ {\bf discrete holomorphic at
$\bm{z}$} iff $[\dopa H](z)=0$, which is equivalent to say that
\begin{equation}
\label{DefHolEquiv} 2[\dpa (H|_\G)](z)=\frac{H(u^{+})\!-\!H(u^{-})}{{u}^{+}\!-\!{u}^{-}} =
\frac{H(w^{+})\!-\!H(w^{-})}{{w}^{+}\!-\!{w}^{-}}=2[\dpa (H|_{\G^*})](z).
\end{equation}
\end{definition}

These difference operators naturally discretize the standard differential operators $\pa h=
\frac{1}{2}(h'_x -ih'_y)$ and $\ol\pa h = \frac{1}{2}(h'_x+ih'_y)$. In particular, $\dpa$ and
$\dopa$ have approximation properties similar to those in Lemma~\ref{ApproxLemma}. Namely,
\[
\left|[\dpa \phi|_\L](z)-(\pa \phi)(z)\right|\ ,\ \left|[\dopa \phi|_\L
](z)-(\ol{\pa}\phi)(z)\right|\ =\ O(\mesh^2)
\]
for smooth functions $\phi$.

Further, for $z\in\DS$, let $\theta_z$ denote the half-angle of the corresponding rhombus
$u^{-}w^{-}u^{+}w^{+}$ along the diagonal $u^{-}u^{+}$, so that
\[
w^{+}\!-\!w^{-}=i\tan\theta_z\cdot (u^{+}\!-\!u^{-}).
\]
We define the {\it weight} of $z$ by
\[
\weightDS{z} :=
\Area(u^{-}w^{-}u^{+}w^{+})=\mesh^2\sin2\theta_z.
\]
Also, for $v\in\G$ and, in the same way, for $v\in\G^*$, we set (cf. (\ref{WeightGDef}))
\[
\weightL{v}:=\frac{1}{4}\sum_{z_s\sim v}\weightDS{z_s}=\frac{\weightG{v}}{2}.
\]
Clearly, formulas similar to (\ref{DiscrIntApprox}) are fulfilled for $\phi$'s defined on
subsets of $\DS$ or $\L$. It is easy to check that definition \ref{DefHolL} may be rewritten
in the following form:
\[
[\dpa H](z)= \frac{1}{4\weightDS{z}}\sum_{v=u^{\pm}\!,\,w^{\pm}}
\ol{\weightE{z}{v}}H(v),\qquad [\dopa H](z)=
\frac{1}{4\weightDS{z}}\sum_{v=u^{\pm}\!,\,w^{\pm}} \weightE{z}{v}H(v),
\]
where the weights $\weightE{z}{v}$ are given by
\[
\weightE{z}{u^{\pm}}:=2\tan\theta_z\cdot (u^{\pm}\!-\!z)=i\cdot(w^{\mp}\!-\!w^{\pm}),
\]
\[
\weightE{z}{w^{\pm}}:=2\cot\theta_z\cdot (w^{\pm}\!-\!z)=i\cdot(u^{\pm}\!-\!u^{\mp}).
\]

The difference operators $\dpa$ and $\dopa$ given above map functions defined on $\L$ to
functions on $\DS$. Further, we introduce their formal adjoint $-(\dpa)^*$, $-(\dopa)^*$, also
denoted by $\dopa$ and $\dpa$, respectively, to keep the notation short. Note that no
confusion arises since the latter operators, vice versa, map functions defined on $\DS$ to
functions on $\L$.

\begin{definition} Let a function $F$ be defined on some subset of $\DS$. For $v\in\L$, we set
\[
[\dopa F](v):= -\frac{1}{4\weightL{v}}\sum_{z_s\sim v}\weightE{z_s}{v} F(z_s)
\quad\mathit{and}\quad [\dpa F](v):= -\frac{1}{4\weightL{v}}\sum_{z_s\sim
v}\ol{\weightE{z_s}{v}} F(z_s),
\]
if the right hand sides make sense. We call $F$ {\bf discrete holomorphic at $\bm{v}$} iff
$[\dopa F](v)=0$.
\end{definition}

These definitions are natural discretization of the formulas
\[
(\ol{\pa}\phi)(v)\approx \frac{\iint_{W(v)}(\ol{\pa} \phi)(x\!+\!iy)dxdy}{\Area(W(v))}  =
-\frac{i}{2\Area(W(v))}\oint_{\pa W(v)} \phi(\z)d\z,
\]
\[
({\pa}\phi)(v)\approx \frac{\iint_{W(v)} (\pa \phi)(x\!+\!iy)dxdy}{\Area(W(v))} =
\frac{i}{2\Area(W(v))}\oint_{\pa W(v)} \phi(\z)d\ol{\z},
\]
where $W(v)$ denotes the corresponding dual face (e.g., see Fig.~\ref{Fig:Notations}B, if
$v=u\in\G$). For constant and linear $\phi$'s, these discretizations give the true answers,
thus
\[
\left|[\dopa \phi|_\DS](v)-(\ol{\pa}\phi)(v)\right|\ ,\ \left|[\dpa \phi|_\DS](v)-(\pa
\phi)(v)\right|\ =\  O(\mesh)
\]
for all smooth functions $\phi$. Note that, in general, one cannot replace $O(\mesh)$ by
$O(\mesh^2)$.

\subsection{Factorization of $\bm\D^\mesh$. Basic properties of discrete holomorphic functions.}
\label{SectElemFactsHol}

The following factorization of $\D^\mesh$ was noted in \cite{Mer01} and \cite{Ken02}:

\begin{proposition}
\label{LaplFactor} For functions $H$ defined on subsets of $\L$ the following is fulfilled:
\[
[\D^\d H](u) = 4[\dpa \dopa H](u) = 4[\dopa \dpa H] (u)
\]
at all vertices $u\in\L$ where the right-hand side makes sense.
\end{proposition}

\begin{proof} Straightforward computations give (see Fig.~\ref{Fig:Notations}B for notations)
\[
[\dopa\dpa H](u) = \frac{1}{8\weightL{u}}\sum_{s=1}^k \left[\tan\theta_s\cdot[H(u_s)\!-\!H(u)]
- i\cdot [H(w_{s+1})\!-\!H(w_s)]\right]= \frac{[\D^\d H](u)}{4}
\]
and similarly for $[\dpa\dopa H](u)$.
\end{proof}

In Lemmas \ref{LemmaHolL}--\ref{LemmaHolBF} below we list basic properties of discrete
holomorphic functions coming from this factorization of $\D^\mesh$. We often omit the word
``discrete'' (e.g., writing ``holomorphic on $\DS$'' instead of ``discrete holomorphic on
$\DS$'') for short.

\begin{lemma}
\label{LemmaHolL} (i) Let a function $H$ be defined on some subset of $\L$. If $H$ is
holomorphic on $\L$, then $H$ is harmonic on both $\G$ and $\G^*$, i.e. both components
$H|_{\G}$, $H|_{{\G^*}}$ are complex-valued harmonic functions.

\smallskip

\noindent (ii) Conversely, in simply connected domains, $H$ is (complex-valued) harmonic on
$\G$ iff there exists a (complex-valued) harmonic on $\G^*$ function $\wt{H}$ such that
$H+i\wt{H}$ is holomorphic on $\L$. $\wt{H}$ is called {\bf discrete harmonic conjugate} to
$H$ and is defined uniquely up to an additive constant. Moreover, $\wt{H}$ is real-valued, if
$H$ is real-valued.
\end{lemma}

\begin{proof}
(i) The claim easily follows by writing $\D^\mesh H = 4\dpa\dopa H =0$.

\noindent (ii) For any $u\in\G$ and $z_s\in\DS$, $z_s\sim u$ (see Fig.~\ref{Fig:Notations}B
for notations), the holomorphicity condition at $z_s$ defines the increments
$\wt{H}(w_{s+1})\!-\!\wt{H}(w_s)$ uniquely. These increments are locally consistent, i.e.
their sum around $u$ is zero, iff $[\D^\mesh H](u)=0$. In simply connected domains, the local
consistency directly implies the global one.
\end{proof}

Due to Lemma \ref{LemmaHolL}, each holomorphic on $\L$ function is a couple of a
complex-valued harmonic function $H|_\G$ and its harmonic conjugate $H|_{\G^*}$. Since the
real part of $H|_\G$ depends only on the imaginary part of $H|_{\G^*}$ (and vice versa), both
functions
\begin{equation}
\label{DefBWonL} \cB H:=\Re H|_\G \!+\!i\Im H|_{\G^*}\quad\mathrm{and}\quad \cW H:=i\Im
H|_\G\!+\!\Re H|_{\G^*}
\end{equation}
are still holomorphic on $\L$ and completely independent of each other. Thus, to avoid a
``doubling of information'', at least unless some boundary conditions are specified, it is
natural to consider (as many authors do) only those $H$, which are purely real on $\G$ (black
vertices of $\L$) and purely imaginary on $\G^*$ (white vertices of $\L$), or vice versa.

\begin{lemma}
\label{LemmaHolDS} (i) Let $H$ be a (complex-valued) harmonic function defined on some subset
of $\G$ or $\G^*$. Then its derivative $F=\dpa H$ is holomorphic on $\DS$ (recall that,
defining $\dpa H$, we formally set $H|_{\G^*}:=0$ or $H|_\G:=0$, respectively). The same holds
true, if $H$ is a holomorphic function defined on some subset of $\L$.

\smallskip

\noindent (ii) Conversely, in simply connected domains, if $F$ is holomorphic on $\DS$, then
there exists a holomorphic on $\L$ function $H$ (which we call {\bf discrete primitive}
$\bm{\int^\mesh F(z)d^\mesh z}$) such that $\dpa H =F$. Its complex-valued harmonic components
$H|_\G$ and $H|_{\G^*}$ are defined uniquely up to (different) additive constants by
\[
H(v^{+})-H(v^{-}):=F(z)\cdot (v^{+}\!-\!v^{-}),\quad
z={\textstyle\frac{1}{2}}(v^{-}\!+\!v^{+}),
\]
where $v^{\pm}\in\G$ or $v^{\pm}\in\G^*$ are neighbors of $z\in\DS$.
\end{lemma}
\begin{proof}
(i) The claim easily follows by writing $\dopa F =\dopa\dpa H=\frac{1}{4}\D^\mesh H =0$.

\noindent (ii) Since we are looking for holomorphic $H$'s, it's necessary and sufficient to
have $\dpa (H|_\G)=\dpa (H|_{\G^*})=\frac{1}{2}F$ (see (\ref{DefHolEquiv})). Thus, the
increments $H(v^{+})\!-\!H(v^{-})$ are defined uniquely. For any $u\in\L$, the condition
$[\dopa F](u)=0$ guarantees that these increments are locally consistent (i.e., their sum
around $u$ is zero). In simply connected domains, this implies the global consistency as well.
\end{proof}

Due to Lemma \ref{LemmaHolDS}, there is a correspondence between holomorphic on $\DS$
functions and their primitives, which are complex-valued harmonic functions on $\G$ (and, in
the same way, on $\G^*$). Since the latter space is naturally split on purely real and purely
imaginary functions, the same should take place for functions, holomorphic on $\DS$.

\begin{definition} Let $z\in\DS$ be the center of the rhombus $u^{-}w^{-}u^{+}w^{+}$, where
$u^{\pm}\in\G$ and $w^{\pm}\in\G^*$, and $F$ be a complex-valued function defined at $z$. We
set
\[
[\cB F](z):=\Pr\left[F(z); \ol{u^{+}\!-\!u^{-}}\,\right]\quad \mathit{and}\quad [\cW
F](z):=\Pr\left[F(z); \ol{w^{+}\!-\!w^{-}}\,\right],
\]
where
\[
\Pr[F;\xi]:=\Re\lt(F\frac{\ol{\xi}}{|\xi|}\rt)\frac{\xi}{|\xi|}=\frac{F\!+\!\ol{F}\xi^2}{2|\xi|^2}
\]
denotes the orthogonal projection of $F$ onto the
line $\xi\R$. Note that $|\cB F|, |\cW F|\le |F|$ and $F=\cB F +\cW F$, since
$u^{+}\!-\!u^{-}\perp w^{+}\!-\!w^{-}$.
\end{definition}

\begin{remark}
\label{RemarkF=BF} Let $F=\dpa H$, where $H$ is purely real on~$\G$ and purely imaginary
on~$\G^*$, or, vice versa, $\Re H|_\G=0$ and $\Im H|_{\G^*}=0$. Then, $F = \cB F$ or $F = \cW
F$, respectively.
\end{remark}

The next Lemma shows that, exactly as it happens for holomorphic on $\L$ functions, each
holomorphic on $\DS$ function $F$ consists of two completely independent halves: $\cB F$ and
$\cW F$, the first coming as a gradient of a real-valued harmonic on $\G$ function and the
second as a gradient of a real-valued harmonic on $\G^*$ function.

\begin{lemma}
\label{LemmaHolBF} A function $F$ is holomorphic on some subset of $\DS$ if and only if both
projections $\cB F$ and $\cW F$ are holomorphic on this subset. Moreover, in this case,
\[
\textstyle \cB F= \dpa \left[\cB \left[\int^\mesh F(z)d^\mesh z\right]\right]\quad
\mathit{and}\quad \cW F= \dpa \left[\cW \left[\int^\mesh F(z)d^\mesh z\right]\right],
\]
where $H=\int^\mesh F(z)d^\mesh z$ is any (local) primitive of $F$ and \mbox{$\cB H$, $\cW H$}
are given by (\ref{DefBWonL}).
\end{lemma}

\begin{proof} It is easy to check that
\[
\dopa[\cB F]=\Re[\dopa F]\ \ \mathrm{and}\ \ \dopa[\cW F]=i\Im[\dopa F]\ \ \mathrm{on}\
\G^{\phantom{*}},
\]
\[
\dopa[\cB F]=i\Im[\dopa F]\ \ \mathrm{and}\ \ \dopa[\cW F]=\Re[\dopa F]\ \ \mathrm{on}\ \G^*,
\]
thus $F$ is holomorphic iff both $\cB F$ and $\cW F$ are holomorphic. In this case, the
primitive $H$ is locally well-defined (up to additive constants), $F= \dpa H = \dpa[\cB H] +
\dpa [\cW H]$, and so $\cB F = \dpa[\cB H]$, $\cW F =\dpa[\cW H]$ (see (\ref{DefBWonL}) and
Remark~\ref{RemarkF=BF}).
\end{proof}

It is worthwhile to note that there exists a natural {\bf averaging operator $\bm{m^\mesh}$},
which maps functions defined on $\L$ to functions on $\DS$. Namely, $m^\mesh$ is given by
\begin{equation}
\label{mMeshLdef} [m^\mesh H](z):= {\textstyle
\frac{1}{4}}\,[H(u^{-})\!+\!H(w^{-})\!+\!H(u^{+})\!+\!H(w^{+})],\quad z\in\DS,
\end{equation}
where, as above, $u^{\pm}\in\G$ and $w^{\pm}\in\G^*$ denote neighbors of $z\in\DS$.

\begin{lemma}
Let $H$ be holomorphic on (some part of) $\L$. Then the averaged function $m^\mesh H$ is
holomorphic on $\DS$ at all $u\in\L$, where the expression $[\dopa m^\mesh H](u)$ makes sense.
\end{lemma}

\begin{proof}

\noindent The condition $[\dopa H](z_s)=0$ (see Fig. \ref{Fig:Notations}B for notations)
implies
\[
[m^\mesh H](z_s)=\frac{H(u)}{2} +
\frac{H(w_{s+1})(w_{s+1}\!-\!u)-H(w_s)(w_s\!-\!u)}{2(w_{s+1}\!-\!w_s)}\,.
\]
Summing the terms $(w_{s+1}\!-\!w_s)[m^\mesh H](z_s)$ around $u$, one arrives at $[\dopa
m^\mesh H](u)=0$.
\end{proof}

Below we will also need the averaging operator $m^\mesh$ (adjoint to (\ref{mMeshLdef})) which,
conversely, maps functions defined on $\DS$ to functions on $\L$:
\begin{equation}
\label{mMeshDSdef} [m^\mesh F](v):= \frac{1}{4\weightL{v}}\sum_{v\sim
z_s\in\DS}\weightDS{z_s}F(z_s),\quad v\in\L.
\end{equation}

Unfortunately, there are two {\it unpleasant facts} that make discrete complex analysis on
rhombic lattices more complicated than the standard continuous theory and even than the square
lattice discretization:
\begin{itemize}

\item One cannot (pointwise) multiply discrete holomorphic functions:
the product $FG$ is {\it not} necessary holomorphic if both $F$ and $G$ are holomorphic.

\item One cannot differentiate discrete holomorphic functions infinitely many times. Moreover,
we don't know any ``local'' discretizations of $\pa$ that map holomorphic functions on $\L$ or
$\DS$ to holomorphic functions defined on the same set ($\L$ or $\DS$). One cannot use natural
combinations of $\dpa$ and $m^\mesh$ since both $\dpa F$ and $m^\mesh F$ are {\it not}
necessary exact holomorphic on $\L$, if $F$ is holomorphic on $\DS$.
\end{itemize}
The first obstacle (multiplication) exists in all discrete theories. Concerning the second,
note that in our case there is some ``nonlocal'' discrete differentiation (so-called dual
integration, see \cite{Duf68} and \cite{Mer07}). Also in two particular cases the local
differentiation leads to holomorphic function again: for the classical definition on the
square grid (since in this case both $\L$ and $\DS$ are square grids, see the book by
J.~Lelong-Ferrand \cite{L-F55}) and for some particular definition on the triangular lattice
(see \cite{DN03}).

\subsection{The Cauchy kernel. The Cauchy formula. Lipschitzness}

The following asymptotic form of the discrete Cauchy kernel is due to R.~Kenyon.

\begin{theorem}[\bf Kenyon]
\label{ThmKenyonHol} Let $z_0\in\DS$. There exists a unique function $F=K(\ccdot;z_0):\L\to\C$
such that
\begin{quotation}
\noindent (i)\phantom{i} $[\dopa F](z)=0$ for all $z\ne z_0$ and $[\dopa F](z_0)\cdot
\weightDS{z_0}=1$;

\smallskip

\noindent (ii) $|F(u)|\to 0$ as $|u-z_0|\to\iy$.
\end{quotation}
Moreover, the following asymptotics hold:
\[
K(u;z_0)=\frac{2}{\pi}\Pr\left[\frac{1}{u\!-\!z_0}\,;\,\ol{u_0^{+}\!-\!u_0^{-}}\,\right]+
O\lt(\frac{\mesh}{|u\!-\!z_0|^2}\rt),\quad u\in\G;
\]
\[
K(w;z_0)=\frac{2}{\pi}\Pr\left[\frac{1}{w\!-\!z_0}\,;\,\ol{w_0^{+}\!-\!w_0^{-}}\,\right]+
O\lt(\frac{\mesh}{|w\!-\!z_0|^2}\rt),\quad w\in\G^*,
\]
where $u_0^{\pm}\in\G$ and $w_0^{\pm}\in\G^*$ are the black and white neighbors of $z_0$,
respectively.
\end{theorem}

\begin{proof}
We give a short sketch of Kenyon's arguments \cite{Ken02} in
Appendix~\ref{SectAKenyonFreeGreen}.
\end{proof}

Let $\O^\d_\G$ be a bounded simply connected discrete domain (see Fig.~\ref{Fig:DiscrDomain}A,
\ref{Fig:DiscrDisc}A). Denote by $B=u_0u_1u_2..u_n$, $u_s\in\G$, its closed polyline boundary,
enumerated in counter clockwise order. Denote by $W=w_0w_1w_2..w_m$, $w_s\in\G^*$, the closed
polyline path (enumerated in counter clockwise order) passing through the centers of all faces
touching $B$ from inside. For functions $G$ defined on $B_\DS:=\DS\cap B$ and $W_\DS:=\DS\cap
W$, we introduce ``discrete contour integrals''
\[
\oint^\mesh_{B} G(z)d^\mesh z := \sum_{s=0}^{n-1}
G\left({\textstyle\frac{1}{2}}(u_{s+1}\!+\!u_s)\right)\cdot(u_{s+1}\!-\!u_s),
\]
\[
\oint^\mesh_{W} G(z)d^\mesh z := \sum_{s=0}^{m-1}
G\left({\textstyle\frac{1}{2}}(w_{s+1}\!+\!w_s)\right)\cdot(w_{s+1}\!-\!w_s).
\]
We also set $\O^\mesh_\L:=\L\cap\O^\mesh$,
\[
\O^\mesh_\DS:=\DS\cap \O^\mesh,\quad \ol{\O}^\mesh_\DS:=\O^\mesh\cup B_\DS\ \ \mathrm{and}\ \
\Int\O^\mesh_\DS:=\O^\mesh_\DS\setminus W_\DS,
\]
where $\O^\mesh$ denotes the polygonal representation of $\O^\mesh_\G$.

\begin{proposition}[\bf Cauchy formula]
\label{CauchyFormula} Let $F:\ol{\O}^\mesh_\DS\to\C$ be a discrete holomorphic function, i.e.,
$[\ol{\pa}^\mesh F](v)=0$ for all $v\in \O^\mesh_\L$. Then, for any $z_0\in \Int\O^\mesh_\DS$,
\[
F(z_0)=\frac{1}{4i}\lt[\oint^\d_{B} K(w(z);z_0)F(z)d^\d z + \oint^\d_{W} K(u(z);z_0)F(z)d^\d
z\rt],
\]
where $w(z)\in W_{\G^*}:=\G^*\cap W$ denotes the nearest ``white'' vertex to $z\in B_\DS$, and
$u(z)\in B_\G:=\G\cap B$ denotes the nearest ``black'' vertex to $z\in W_\DS$ (see
Fig.~\ref{Fig:DiscrDisc}A).
\end{proposition}

\begin{proof}
By definitions of the discrete Cauchy kernel $K$ and the operator $\dopa$, one has
\[
4F(z_0)\ = \!\!\!\!\!\mathop{\sum\sum}\limits_{z\in \O^\mesh_\DS,\ z\sim v,\ v\in
\ol{\O}^\mesh_\L}\!\!\!\!\! F(z)\weightE{z}{v}K(v;z_0)\ =
\!\!\!\!\!\mathop{\sum\sum}\limits_{v\in \ol{\O}^\mesh_\L,\ v\sim z,\ z\in\O^\d_\DS}\!\!\!\!\!
K(v;z_0)\weightE{z}{v}F(z),
\]
where $\ol{\O}^\mesh_\L:=\O^\mesh_\L\cup B_\G$. Since $\sum_{v\sim z,\
z\in\ol{\O}^\mesh_\DS}\weightE{z}{v}F(z)=0$ for all $v\in\O^\mesh_\L$, this gives
\begin{align*}
4F(z_0)\ &= \!\!\!\!\!\sum_{v\in B_\G,\ v\sim z,\ z\in \O^\mesh_\DS}\!\!\!\!\!
K(v;z_0)\weightE{z}{v}F(z)\ \ - \!\!\!\!\!\sum_{v\in \O^\mesh_\DS,\ v\sim z,\ z\in
B_\DS}\!\!\!\!\! K(v;z_0)\weightE{z}{v}F(z)\cr &=\ \sum_{z\in W_\DS}
K(u(z);z_0)F(z)\weightE{z}{u(z)}\ -\ \sum_{z\in B_\DS} K(w(z);z_0)F(z)\weightE{z}{w(z)}.
\end{align*}
Both sums coincide with the discrete contour integrals defined above.
\end{proof}

The Cauchy formula may be nicely rewritten in the asymptotic form for both components $\cB F$
and $\cW F$ of a holomorphic function $F$ separately. Recall that these components are
completely independent of each other (see Lemma~\ref{LemmaHolBF}).

\begin{corollary}[\bf asymptotic Cauchy formula]
\label{CauchyFormBF} Let $F:\ol{\O}^\mesh_\DS\to\C$ be a discrete holomorphic function,
$z_0\in\Int\O^\mesh_\DS$ and $u_0^{\pm}\in \G$, $w_0^{\pm}\in\G^*$ be its neighboring
vertices. Then
\[
[\cB F](z_0) = \Pr\lt[\frac{1}{2\pi i}\lt(\oint^\mesh_B\frac{[\cB F](z)}{z\!-\!z_0}\,d^\mesh z
+ \oint^\mesh_W\frac{[\cB F](z)}{z\!-\!z_0}\,d^\mesh z\rt)\ ;\ \ol{u_0^{+}\!-\!u_0^{-}}\,\rt]
+ O\lt(\frac{\mesh M L} {d^2}\rt),
\]
where $d=\dist(z_0,W)$, $M=\max_{z\in B_\DS\cup W_\DS}|F(z)|$ and $L=\Length(B)+\Length(W)$.
The same formula holds true for $\cW F$, if one replaces $u_0^{+}\!-\!u_0^{-}$ by
$w_0^{+}\!-\!w_0^{-}$.
\end{corollary}

\begin{proof}
We plug Kenyon's asymptotics (Theorem~\ref{ThmKenyonHol}) into Proposition
\ref{CauchyFormula}:

\noindent if $z\in W_\DS$, then $[\cB F](z)d^\mesh z/4i\in \R$, and so
\[
K(u(z);z_0)\cdot \frac{[\cB F](z)d^\mesh z}{4i} = \Pr\lt[\frac{[\cB F](z)d^\mesh z} {2\pi
i(z\!-\!z_0)}\ ;\, \ol{u_0^{+}\!-\!u_0^{-}}\,\rt]+O\lt(\frac{\mesh M |d^\mesh z|} {d^2}\rt);
\]

\noindent if $z\in B_\DS$, then $[\cB F](z)d^\mesh z/4i\in i\R$, and so, again,
\[
K(w(z);z_0)\cdot \frac{[\cB F](z)d^\mesh z}{4i} = \Pr\lt[\frac{[\cB F](z)d^\mesh z} {2\pi
i(z\!-\!z_0)}\ ;\, \ol{u_0^{+}\!-\!u_0^{-}}\,\rt]+O\lt(\frac{\mesh M |d^\mesh z|} {d^2}\rt),
\]
since $w_0^{+}\!-\!w_0^{-}\perp u_0^{+}\!-\!u_0^{-}$. The claim follows by summing along $B$
and $W$.
\end{proof}

Finally, the Cauchy formula implies Lipschitzness of discrete holomorphic functions. Since
$\cB F$ and $\cW F$ are independent of each other, this should be valid for both components
separately. On the other hand, the phase of $[\cB F](z)$ depends only on the direction of the
edge $u^{-}u^{+}$ passing through $z$, so one cannot expect that $[\cB F](z_1)$ and $[\cB
F](z_2)$ are close in the usual sense, if $z_1$ and $z_2$ are close. Thus, we firstly use the
operator $m^\mesh$ defined by (\ref{mMeshDSdef}) and average our function around vertices
$v\in\L$.

\begin{proposition}[\bf Lipschitzness of discrete holomorphic functions]
\label{PropLipHol} Let \mbox{$u\in\G$} and let $F$ be discrete holomorphic in
$\ol{B}^\mesh_\DS(u,R)$. Then, for all $z_s\sim u$, $z_s\in\DS$ (see Fig.~\ref{Fig:Notations}B
for notations),
\[
\left|[\cB F](z_s)-\Pr\left[2[m^\mesh (\cB F)](u);\ol{u_s\!-\!u}\,\right]\right|\le
\const\cdot \frac{M\mesh}{R}, \quad \mathit{where}\quad M=\max_{\ol{B}^\mesh_\DS(u,R)}|F(z)|.
\]
The same formula holds true for $\cW F$, if one replaces $u_s-u$ by $w_{s+1}-w_s$.
Furthermore, if $v_1,v_2\in B^\mesh_\L(u,r)$, $r<R$, then
\[
\left|[m^\mesh F](v_2)-[m^\mesh F](v_1)\right|\le \const\cdot\frac{M|v_2\!-\!v_1|}{R-r}.
\]
\end{proposition}

\begin{proof}
Let $B$ and $W$ be the same discrete contours as above (see Fig.~\ref{Fig:DiscrDisc}A), note
that their lengths are bounded by $\const\cdot R$. Applying Corollary~\ref{CauchyFormBF} for
all $z_s\sim u$ and taking into account that $|(z\!-\!z_s)^{-1}-(z\!-\!u)^{-1}|\le \const\cdot
\mesh/R^2$, one obtains
\[
[\cB F](z_s)=\Pr[A;\ol{u_s\!-\!u}\,]+O\lt(\frac{M\mesh}{R}\rt),\quad A:= \frac{1}{2\pi
i}\lt(\oint^\mesh_B\frac{[\cB F](z)}{z-u}\,d^\mesh z + \oint^\mesh_W\frac{[\cB
F](z)}{z-u}\,d^\mesh z\rt).
\]
Due to the identity
\[
\frac{1}{4\weightL{u}}\sum_{z_s\sim u} \weightDS{z_s}\Pr[A;\ol{u_s\!-\!u}\,]=
\frac{1}{4\weightL{u}}\sum_{z_s\sim u} \mesh^2\sin2\theta_s\cdot
\frac{A+e^{-2i\arg(u_s-u)}\ol{A}}{2}
\]
\[
=\frac{A}{2}+ \frac{\mesh^2\ol{A}}{16i\weightL{u}}\sum_{u_s\sim
u}(e^{-2i\arg(w_s-u)}-e^{-2i\arg(w_{s+1}-u)})=\frac{A}{2},
\]
it gives
\[
[m^\mesh (\cB F)](u) = \frac{A}{2} + O\lt(\frac{M\mesh}{R}\rt).
\]
In particular, $\left|[\cB F](z_s)-\Pr[2[m^\mesh (\cB F)](u);\ol{u_s\!-\!u}]\right|
\le\const\cdot M\mesh/R$. The proof for $\cW F$ goes exactly in the same way, since
$e^{-2i\arg(w_{s+1}-w_s)}=-e^{-2i\arg(u_s-u)}$. Moreover, using the same calculations for
$[m^\mesh F](w_s)$, one obtains
\[
\left|[m^\mesh (\cB F)](w_s) - [m^\mesh (\cB F)](u)\right|\ ,\ \left|[m^\mesh (\cW F)](w_s) -
[m^\mesh (\cW F)](u)\right| \le \const \cdot \frac{M\mesh}{R},
\]
so the same estimate holds true for the function $m^\mesh F = m^\mesh(\cB F)+ m^\mesh(\cW F)$.

Summing these inequalities along the path connecting $v_1$ and $v_2$ inside $B^\mesh_\G(u,r)$
(due to condition \sps, there is a path of length $\le\const\cdot\mesh^{-1}|v_2\!-\!v_1|$),
one immediately arrives at the estimate for $|[m^\mesh F](v_2) - [m^\mesh F](v_1)|$.
\end{proof}

\section{Convergence theorems}
\label{SectConvThms} \setcounter{equation}{0}

\subsection{Precompactness in the \mbox{$\bm{C^1}$-topology}.}
\label{SectPrecompactnessC1}

In the continuous setup, each uniformly bounded family of harmonic functions (defined in some
common domain $\O$) is precompact in the \mbox{$C^{\iy}$-topology}. Using
Corollary~\ref{CorLipHarm} and Proposition~\ref{PropLipHol}, it is easy to prove the analogue
of this statement for discrete harmonic functions.

Below we widely use the following convention. Let a function $H^\mesh$ be defined in a
discrete domain $\O^\mesh_\G\ss\G^\mesh$. Then, $H^{\mesh}$ can be thought of as defined in
its polygonal representation $\O^{\mesh}\ss\C$ by some standard continuation procedure, say,
linear on edges and harmonic inside faces. Note that this continuation is bounded in
$\O^\mesh$, if $H^\mesh$ is bounded on $\O^\mesh_\G$, and Lipschitz in $\O^\mesh$, if
$H^\mesh$ is Lipschitz on $\O^\mesh$(with the same constants).

\begin{proposition}
\label{PropUniformComp} Let $H^{\mesh_j}:\O^{\mesh_j}_\G\to\R$ be (real-valued) discrete
harmonic functions defined in discrete domains $\O^{\mesh_j}_{\G}\ss\G^{\mesh_j}$ with
$\mesh_j\to 0$. Let $\O\ss\bigcup_{n=1}^{+\iy}\bigcap_{j=n}^{+\iy}\O^{\mesh_j}\ss\C$ be some
continuous domain. If $H^{\mesh_j}$ are uniformly bounded on $\O$, i.e.
\[
\max\nolimits_{u\,\in\, \O^{\mesh_j}_\G\cap\,\O}|H^{\mesh_j}(u)|\le M < +\iy\quad \mathit{for\
all}\ \ j,
\]
then there exists a subsequence $\mesh_{j_k}\to 0$ (which we denote by $\mesh_k$ for short)
and two functions $h:\O\to\R$, $f:\O\to\C$ such that (we denote by ``$\rra$'' uniform
convergence)
\[
H^{\mesh_k}\rra h\quad\mathit{uniformly~on\ compact\ subsets}\ K\ss\O
\]
and
\begin{equation}
\label{GradConv} \frac{H^{\mesh_k}(u_k^{+})-H^{\mesh_k}(u_k^{-})}{|u_k^{+}-u_k^{-}|}\rra
\Re\lt[f(u)\cdot \frac{u_k^{+}-u_k^{-}}{|u_k^{+}-u_k^{-}|}\rt],
\end{equation}
if $u_k^{\pm}\in\G^{\mesh_k}$, $u_k^{+}\sim u_k^{-}$ and $u_k^{\pm}\to u\in K\ss\O$ as
$k\to\infty$. Moreover, the limit function~$h$, $|h|\le M$, is harmonic in $\O$ and $f=
h'_x-ih'_y= 2\pa h$ is analytic in $\O$.
\end{proposition}

\begin{remark}
In other words, the discrete gradients of $H^\mesh$ defined by the left-hand side of
(\ref{GradConv}) converge to $\nabla h$. Looking at the edge $u_k^{-}u_k^{+}$ one sees only
the discrete directional derivative of $H^\mesh$ along the unit vector
$\t_k:=(u_k^{+}-u_k^{-})/|u_k^{+}-u_k^{-}|$ which converges to $\langle \nabla
h(u),\t_k\rangle = \Re[2\pa h(u)\cdot\t_k]$.
\end{remark}

\begin{proof}
Due to the uniform Lipschitzness of bounded discrete harmonic functions (see Corollary
\ref{CorLipHarm}) and the Arzel\`a-Ascoli Theorem, the sequence $\{H^{\mesh_j}\}$ is
precompact in the uniform topology on any compact subset $K\ss\Omega$. Moreover, their
discrete derivatives (defined for $z\in\O^{\mesh_j}_\DS$)
\[
F^{\mesh_j}(z):=[\pa^{\mesh_j} H^{\mesh_j}](z)=
\frac{H^{\mesh_j}(u_j^{+}(z))-H^{\mesh_j}(u_j^{-}(z))}{u_j^{+}(z)-u_j^{-}(z)},\qquad z\sim
u_j^{\pm}(z)\in \G^{\mesh_j},
\]
are discrete holomorphic and uniformly bounded on any compact subset $K\ss \O$. Then, due to
Proposition \ref{PropLipHol} and the Arzel\`a-Ascoli Theorem, the sequence of averaged
functions $m^{\mesh_j} F^{\mesh_j}$ (defined on $\O^{\mesh_j}_\G$ by (\ref{mMeshDSdef})) is
precompact in the uniform topology on any compact subset of $\Omega$. Thus, for some
subsequence $\mesh_k\to 0$, one has
\[
H^{\mesh_k}\rra h\qquad \mathrm{and} \qquad 2m^{\mesh_k}F^{\mesh_k}\rra f
\]
uniformly on compact subsets of $\O$. Moreover, due to Proposition~\ref{PropLipHol}, it also
gives
\[
\left|F^{\mesh_k}(z) - \Pr\left[f(z);\ol{u_k^{+}(z)-u_k^{-}(z)}\,\right]\right|\rra 0\quad
\mathrm{uniformly\ on\ compact\ subsets\ of}\ \O.
\]
It is easy to see that $h$ is harmonic. Indeed, let $\phi:\O\to\R$ be an arbitrary
$C_0^\iy(\O)$ test function (i.e., $\phi\in C^\iy$ and $\supp\phi\ss\O$). Denote by $h^\mesh$,
$\phi^\mesh$ and $(\D\phi)^\mesh$ the restrictions of $h$, $\phi$ and $\D\phi$ onto the
lattice $\G^\mesh$. The approximation properties ((\ref{DiscrIntApprox}) and
\mbox{Lemma~\ref{ApproxLemma}}) and discrete integration by parts give
\[
\langle h,\D\phi\rangle_\O\ = \lim_{\mesh=\mesh_k\to 0} \sum_{u\in \O^\mesh_\G}
h^\mesh(u)(\D\phi)^\mesh(u)\weightG{u}\ = \lim_{\mesh=\mesh_k\to 0} \sum_{u\in \O^\mesh_\G}
h^\mesh(u)[\D^\mesh\phi^\mesh](u)\weightG{u}
\]
\[
= \lim_{\mesh=\mesh_k\to 0} \sum_{u\in \O^\mesh_\G}
H^\mesh(u)[\D^\mesh\phi^\mesh](u)\weightG{u}\ = \lim_{\mesh=\mesh_k\to 0} \sum_{u\in
\O^\mesh_\G} [\D^\mesh H^\mesh](u)\phi^\mesh(u)\weightG{u}\ =\ 0.
\]
Furthermore, for any path $[u_0u_n]^\mesh=u_0u_1..u_n$, $u_{s+1}\sim u_s$, $u_s\in\G^\mesh$,
one has
\[
H^\mesh(u_n)-H^\mesh(u_0)=\int^\mesh_{[u_0u_n]^\mesh} F^\mesh\!(z)d^\mesh z = \sum_{s=0}^{n-1}
F^\mesh\!\left({\textstyle\frac{1}{2}}(u_{s+1}\!+\!u_s)\right)\cdot (u_{s+1}\!-\!u_s).
\]
Taking appropriate discrete approximations of segments $[uv]\ss\O$ (recall that rhombi angles
are bounded from $0$ and $\pi$, so one may find polyline approximations with uniformly bounded
lengths) and passing to the limit as $\mesh=\mesh_k\to 0$, one obtains
\[
h(v)-h(u)= \int_{[uv]} \Re[f(z)dz] = \Re\left[\int_{[uv]} f(z)dz\right]\quad \mathrm{for\ all\
segments}\ [uv]\ss\O.
\]
It gives $\a h'_x(u)+\b h'_y(u)=\Re[(\a+i\b)f(u)]$ for all $u\in \O$ and $\a,\b\in\R$, so
$f=2\pa h$.
\end{proof}

As an illustration of what directly follows from basic facts collected in
Sect.~\ref{SectBasicFacts}, we give a proof of the most classical convergence result for
solutions of the Dirichlet boundary value problem, when a {\it single} domain $\O\ss\C$
bounded by Jordan curves is approximated by discrete ones, ``growing from inside''. Later, in
Theorem~\ref{ThmDirConvUni}, we will prove the {\it uniform} (w.r.t.~$\O$) version of the same
result for simply connected $\O$'s.

\begin{proposition}
\label{PropDirConv} Let $\O\ss\C$ be a (possibly not simply connected) continuous domain,
bounded by a finite number of closed nonintersecting Jordan curves, $\pa\O=J_1\cup..\cup J_n$,
and $g:J^r\to\R$ be a continuous function defined in some closed \mbox{$r$-neighborhood} $J^r$
of~$\pa\O$. Let a sequence of discrete domains $\O^{\mesh_j}_\G\ss\G^{\mesh_j}$, $\mesh_j\to
0$, approximates $\O$ so that
\[
\O\setminus J^r\ss\O^{\mesh_1}\ss\O^{\mesh_2}\ss\dots \ss \O\quad \mathit{and}\quad
\cup_{j=1}^{+\infty}\O^{\mesh_j}=\O.
\]
Let $H^{\mesh_j}$ denote the discrete harmonic continuation of $g$ from $\pa\O^{\mesh_j}_\G\ss
J^r$ into $\O^{\mesh_j}_\G$ and $h$ be the continuous harmonic continuation of $g$ from $J$
into $\O$. Then,
\[
H^{\mesh_j}\rra h\ \mathit{uniformly\ on\ compact\ subsets}\ K\ss\O,
\]
Moreover, discrete gradients (\ref{GradConv}) of functions $H^{\mesh_j}$ uniformly converge to
$\nabla h$.
\end{proposition}

\begin{proof} Since $J^r$ is compact, $g$ is bounded by some constant $M:=\max_{z\in J^r}|g(z)|$ and uniformly
continuous on $J^r$. Set
\[
\textstyle \nu(\r):=\max_{z,w\in J^r:|z-w|\le\r}|g(z)-g(w)|\to 0\ \ \mathrm{as}\ \ \r\to 0.
\]
By the maximum principle, all $H^{\mesh_j}$ are uniformly bounded in $\O$. Then,
Proposition~\ref{PropUniformComp} allows one to extract a subsequence $H^{\mesh_k}$ which
converges to some harmonic function $H$ (and the gradients of $H^{\mesh_k}$ converge to
$\nabla H$). Thus, it is sufficient to prove that each subsequential limit coincides with $h$,
i.e. to identify the boundary values of $H$.

Let $z=z^{\mesh_k}\in\O^{\mesh_k}_\G\ss\G^{\mesh_k}$, $w\in\pa\O$ be (one of) the closest to
$z$ points on $\pa\O$, and $d:=|z\!-\!w|$. Since $H^\mesh=g$ on $\pa\O^\mesh_\G$, for any
$\mesh=\mesh_k$ and $\rho\ge 2d$, one has
\[
\begin{split}
|H^\mesh(z) - g(w)| & \le \max\left\{|H^\mesh(u)\!-\!g(w)|,\ u\in \pa\O^\mesh_\G\cap
B(z,\r)\right\} \cdot \dhm{\mesh}{z}{\pa\O^\mesh_\G\cap B(z,\r)}{\O^\mesh_\G}
 \cr & \phantom{\le}+
\max\left\{|H^\mesh(u)\!-\!g(w)|,\ u\in \pa\O^\mesh_\G\setminus B(z,\r)\right\}\cdot
\dhm{\mesh}{z}{\pa\O^\mesh_\G\setminus B(z,\r)}{\O^\mesh_\G} \cr & \le
\nu(2\r)+2M\cdot\const\cdot (2d/\r)^\beta,
\end{split}
\]
where we have used $\dist(z;\pa\O^\mesh_\G)\le d\!+\!2\mesh\le 2d$ and the weak Beurling-type
estimate (Proposition~\ref{PropWeakBeurling}) for the second discrete harmonic measure.
Choosing $\r(d)$ so that $\nu(2\r(d))\cdot\r(d)^\beta=d^\beta$ and passing to the limit as
$\mesh\to 0$, we obtain the estimate
\[
|H(z) - g(w)|=O(\nu(2\r(d)))\to 0\ \ \mathrm{as}\ \ d=|z\!-\!w|\to 0.
\]
Thus, boundary values of $H$ coincide with those of $h$, hence $H=h$ in $\O$.
\end{proof}

\subsection{Carath\'eodory topology and uniform \mbox{$\bm{C^1}$-convergence}.}

\label{SectCaraConv}

Below we need some standard concepts of geometric function theory (see \cite{Pom92},
Chapters 1,2).

Let $\O$ be a simply connected domain. A {\it crosscut} $C$ of $\O$ is an open Jordan arc in
$\O$ such that $\ol{C}=C\cup\{a,b\}$ with $a,b\in\pa\O$. A {\it prime end} of $\O$ is an
equivalence class of sequences ({\it null-chains}) $(C_n)$ of prime ends such that
$\ol{C}_n\cap\ol{C}_{n+1}=\emptyset$, $C_n$ separates $C_0$ from $C_{n+1}$ and $\diam C_n\to
0$ as $n\to\iy$ (null chains $(C_n)$, $(\wt{C})_n$ are equivalent iff for all sufficiently
large $m$ there exists $n$ such that $C_m$ separates $\wt{C}_0$ from $\wt{C}_n$ and $\wt{C}_m$
separates $C_0$ from $C_n$).

Let $P(\O)$ denote the set of all prime ends of $\O$ and let $\phi:\O\to{\mathbb D}$ be a
conformal map. Then (see Theorem 2.15 in \cite{Pom92}) $\phi$ induces the natural bijection
between $P(\O)$ and the unit circle ${\mathbb T}=\pa{\mathbb D}$.

Let $u_0\in\C$ be given and $\O_n,\O\ss\C$, be simply connected domains $\ne\C$ with
$u_0\in\O_n,\O$. We say that $\O_n\to\O$ as $n\to\iy$ {\it in the sense of kernel convergence
with respect to $u_0$} iff

(i)\phantom{i} some neighborhood of every $z\in\O$ lies in $\O_n$ for large enough $n$;

(ii) for every $a\in\pa\O$ there exist $a_n\in\pa\O_n$ such that $a_n\to a$ as $n\to\iy$.

\noindent Let $\phi_k:\O_k\to\dD$, $\phi:\O\to\dD$ be the Riemann uniformization maps
normalized at $u_0$ (i.e., $\phi(u_0)=0$ and $\phi'(u_0)>0$). Then (see Theorem 1.8
\cite{Pom92})
\begin{center}
$\O_k\to\O$ w.r.t. $u_0$\qquad$\Lra$\qquad $\phi_k^{-1}\rra\phi^{-1}$ uniformly on compact
subsets of $\dD$.
\end{center}
Using the Koebe distortion theorem (see Section 1.3 in \cite{Pom92}), it is easy to see that

(a) $\phi_k\rra\phi$ as $k\to\iy$ uniformly on compact subsets $K\ss\O$;

(b) for any $\O_1,\O_2$ such that $u_0\in\O_1\ss\O_2\ne\C$, the set of all simply connected
domains $\{\O:\O_1\ss\O\ss\O_2\}$ is compact in the topology of kernel convergence
w.r.t.~$u_0$.

\begin{definition}
Let $\O=(\O;v,..;a,b,..)$ be a simply connected bounded domain with several (possibly none)
marked interior points $v,..\in\Int\O$ and prime ends (boundary points) $a,b,..\in P(\O)$ (we
admit coincident points, say, $a\!=\!b$) and let $u\in\O$. We write
\[
(\O_k;u_k)=(\O_k;u_k,v_k,..;a_k,b_k,..)\ \CaraTo\ (\O;u)=(\O;u,v,..;a,b,..)\ \ \mathrm{as}\
k\to\iy,
\]
iff the domains $\O_k$ are uniformly bounded, $u_k\to u$, $\O_k\to\O$ in the sense of kernel convergence
w.r.t.~$u$ and $\phi_k(v_k)\to\phi(v),..$, $\phi_k(a_k)\to\phi(a),..$, where
$\phi_k:\O_k\to{\mathbb D}$, $\phi:\O\to{\mathbb D}$ are the Riemann uniformization maps
normalized at $u$.
\end{definition}

\begin{remark}
Since $v\in\O$, one has $|\phi(v)|<1$. Thus, $\phi_k(v_k)\to\phi(v)$ implies $v_k\to v$.
Moreover, one can equivalently use the point $v$ instead of $u$ in the definition given above.
\end{remark}

\begin{definition}
Let $\O$ be a simply connected bounded domain, $u,v,..\in\O$ and \mbox{$r>0$}. We say that the
inner points $u,v,..$ are {\bf jointly \mbox{$\bm{r}$-inside}~$\bm{\O}$} iff
$B(u,r),B(v,r),..\ss\O$ and there are paths $L_{uv},..$ connecting these points
\mbox{$r$-inside}~$\O$ (i.e., $\dist(L_{uv},\pa\O),..\ge r$). In other words, $u,v,..$ belong
to the same connected component of the \mbox{$r$-interior} of $\O$.
\end{definition}

Note that for each $0<r<R$ there exists some $C(r,R)$ such that, if $\O\ss B(0,R)$ and
$u,v,..$ are jointly \mbox{$r$-inside}~$\O$, then
\begin{equation}
\label{phi(v)Bound} |\phi(v)|,..\le C(r,R)<1,
\end{equation}
where $\phi:\O\to\dD$ is the Riemann uniformization map normalized at $u$. Indeed, considering
the standard plane metric, one concludes that the extremal distance (see, e.g., Chapter IV in
\cite{GM05}) from $L_{uv}$ to $\pa\O$ in $\O\setminus L_{uv}$ is not less than $r/\pi R^2$.
Thus, the conformal modulus of the annulus $\dD\setminus \phi(L_{uv})$ is bounded below by
some $\const(r,R)>0$. Since $\phi(u)=0$, (\ref{phi(v)Bound}) holds true.

Now we formulate a general framework for Theorems~\ref{ThmDirConvUni}--\ref{ThmDPoConv}.
Suppose that some harmonic function (e.g., harmonic measure, Green's function, Poisson kernel
etc.)
\[
h(\ccdot;\O)=h(\ccdot,v,..;a,b,..;\O):\O\to\R
\]
is associated with each (continuous) domain $\O=(\O;v,..;a,b,..)$.

Similarly, let $\O^\mesh_\G=(\O^\mesh_\G;v^\mesh,..;a^\mesh,b^\mesh,..)$ denote a simply
connected bounded discrete domain with several marked vertices $v^\mesh,..\in\Int\O^\mesh_\G$
and $a^\mesh,b^\mesh,..\in\pa\O^\mesh_\G$ and
\[
H^\mesh(\ccdot;\O^\mesh_\G)=
H^\mesh(\ccdot,v^\mesh,..;a^\mesh,b^\mesh,..;\O^\mesh_\G):\O^\mesh_\G\to\R
\]
be some discrete harmonic function associated with this configuration. The idea of
Proposition~\ref{PropUniC1} is to use the compactness argument again, now for the set of all
simply-connected domains. Recall that $\O^\mesh\ss\C$ denotes the polygonal representation of
$\O^\mesh_\G$.

\begin{definition}
\label{DefUniformC1conv} We say that {\bf $\bm{H^\mesh}$ are uniformly
\mbox{$\bm{C^1}$\!-close} to $\bm{h}$ inside~$\bm{\O^\mesh}$}, iff for all $0\!<\!r\!<\!R$
there exist $\ve(\mesh)=\ve(\mesh,r,R),\ \wt{\ve}(\mesh)=\wt{\ve}(\mesh,r,R)\to 0$ as
$\mesh\to 0$ such that for all discrete domains
$\O^\mesh_\G=(\O^\mesh_\G;v^\mesh,..;a^\mesh,b^\mesh,..)$ and $u^\mesh\in\Int\O^\mesh_\G$ the
following holds true:
\begin{quotation}
\noindent If $\O^\mesh\ss B(0,R)$ and $u^\mesh,v^\mesh,..$ are jointly
\mbox{$r$-inside}~$\O^\mesh$, then
\begin{equation}
\label{HUniformClose} \left|H^\mesh(u^\mesh,v^\mesh,..;a^\mesh,b^\mesh,..;\O^\mesh_\G)\,-\,
h(u^\mesh,v^\mesh,..;a^\mesh,b^\mesh,..;\O^\mesh)\right|\le \ve(\mesh)
\end{equation}
and, for all $u^\mesh_s\sim u^\mesh$, $u^\mesh_s\in \O^\mesh_{\G}$,
\begin{equation}
\label{HGradUnifromClose}
\left|\frac{H^\mesh(u^\mesh_s;\O^\mesh_\G)-H^\mesh(u^\mesh;\O^\mesh_\G)}{|u^\mesh_s-u^\mesh|}
\,-\, \Re\lt[2\pa h(u^\mesh;\O^\mesh)\cdot
\frac{u^\mesh_s-u^\mesh}{|u^\mesh_s-u^\mesh|}\rt]\right|\le \wt{\ve}(\mesh),
\end{equation}
where $2\pa h = h'_x-ih'_y$.
\end{quotation}
\end{definition}

\begin{proposition}
\label{PropUniC1} Let (a) $H^\mesh\to h$ ``pointwise'' as $\mesh\to 0$, i.e.,
\begin{equation}
\label{AssumptA_PWConv} H^\mesh(u^\mesh;\O^\mesh_\G)\to h(u;\O),\ \ \mathit{if}\ \
(\O^\mesh;u^\mesh)\ \CaraTo\ (\O;u)\quad \mathit{as}\ \ \mesh\to 0;
\end{equation}
and (b) $h$ be Carath\'eodory-stable, i.e.,
\begin{equation}
\label{AssumptB_CaraStab} h(u_k;\O_k)\to h(u;\O),\ \ \mathit{if}\ \ (\O_k;u_k)\ \CaraTo\
(\O;u)\quad \mathit{as}\ \ k\to\iy.
\end{equation}
Then functions $H^\mesh$ are uniformly \mbox{$C^1$\!-close} to $h$ inside $\O^\mesh$ (see
Definition \ref{DefUniformC1conv}).
\end{proposition}

\begin{remark}
\label{RemA=>B} Typically, if one is able to prove (a) using the ``toolbox'' developed in
Sect.~\ref{SectBasicFacts}, then the same reasoning applied in the continuous setup would lead
to (b), since all these tools are just discrete versions of classical facts from complex
analysis.
\end{remark}

\begin{proof}
Suppose (\ref{HUniformClose}) does not hold true, i.e.,
\[
\left|H^{\mesh}(u^{\mesh};\O^{\mesh}_\G)-h(u^{\mesh};\O^{\mesh})\right|\ge \ve_0>0
\]
for some sequence $(\O^{\mesh}_{\G};u^{\mesh})$, $\mesh=\mesh_j\to 0$, such that
$B(u^{\mesh},r)\ss\O^{\mesh}\ss B(0,R)$. Taking a subsequence, one may assume that
$u^{\mesh}\to u$ for some $u\in B(0,R)$. The set of all simply connected domains
$\O:B(u,\frac{1}{2}r)\ss\O\ss B(0,R)$ is compact in the Carath\'eodory topology. Thus, taking
a subsequence again, one may assume that
\[
(\O^{\mesh};u^{\mesh},v^{\mesh},..;a^{\mesh},b^{\mesh},..)\ \CaraTo\ (\O;u,v,..;a,b,..)\quad
\mathrm{as}\ \ \mesh=\mesh_k\to 0
\]
(note that the marked points $v^{\mesh},..$ cannot reach the boundary due to
(\ref{phi(v)Bound})). Then, (a)~the pointwise convergence $H^\mesh\to h$ and (b)~the
Carath\'eodory-stability of $h$ easily give a contradiction. Indeed, both
\[
\mathrm{(a)}\ H^{\mesh}(u^{\mesh};\O^{\mesh}_\G)\to h(u;\O)\quad \mathrm{and}\quad
\mathrm{(b)}\ h(u^{\mesh};\O^{\mesh})\to h(u;\O)\quad \mathrm{as}\ \ \mesh=\mesh_k\to 0.
\]

In view of Proposition~\ref{PropUniformComp}, the proof for discrete gradients goes by the
same way. Assume (\ref{HGradUnifromClose}) does not hold for some sequence of discrete
domains. As above, one may take a subsequence $\mesh=\mesh_k$ such that
$(\O^{\mesh};u^{\mesh})\CaraTo (\O,u)$. Note that (b) directly implies
\[
h(\ccdot;\O^{\mesh})\rra h(\ccdot;\O)\quad \mathrm{uniformly\ on}\ \ \ol{B}(u,
{\textstyle\frac{1}{2}}r)\quad \mathrm{as}\ \ \mesh=\mesh_k\to 0.
\]
Indeed, $(\O^{\mesh};\wt{u})\CaraTo(\O;\wt{u})$ for all $\wt{u}\in\ol{B}(u,\frac{1}{2}r)$. If
$|h(\wt{u}^\mesh;\O^{\mesh})-h(\wt{u}^\mesh;\O)|\ge\ve_0>0$ for some $\wt{u}^\mesh$ and all
$\mesh=\mesh_k$, then, taking a subsequence $\mesh=\mesh_m$ so that $\wt{u}^{\mesh}\to
\wt{u}\in\ol{B}(u,\frac{1}{2}r)$, one obtains a contradiction with
$h(\wt{u}^{\mesh};\O^{\mesh})\to h(\wt{u};\O)$ and $h(\wt{u}^{\mesh};\O)\to h(\wt{u};\O)$.

The uniform estimate $|H^{\mesh}(\ccdot;\O^{\mesh}_\G)-h(\ccdot;\O^{\mesh})|\le\ve(\mesh)\to
0$ (see above) gives
\[
H^{\mesh}(\ccdot;\O^{\mesh}_\G)\rra h(\ccdot;\O)\quad \mathrm{uniformly\ on}\ \ \ol{B}(u,
{\textstyle\frac{1}{2}}r)\quad \mathrm{as}\ \ \mesh=\mesh_k\to 0.
\]
In particular, all $H^{\mesh_k}(\ccdot;\O^{\mesh_k}_\G)$ are uniformly bounded in
$\ol{B}(u,\frac{1}{2}r)$. Thus, using Proposition~\ref{PropUniformComp}, one can find a
subsequence $\mesh=\mesh_m$ such that the discrete derivatives of $H^{\mesh_m}$ converge (as
defined by (\ref{HGradUnifromClose})) to $f=2\pa h(\ccdot;\O)$ which gives a contradiction.
\end{proof}

\subsection{Basic uniform convergence theorems}
\label{SectMainConvThm}

We start with a uniform (w.r.t.~$\O$) version (Theorem~\ref{ThmDirConvUni}) of
Proposition~\ref{PropDirConv} for simply-connected domains. It immediately gives the uniform
convergence for the discrete Green's functions (Corollary~\ref{CorGreenConv}). Then, we prove
very similar Theorem~\ref{ThmDhmConv} devoted to the discrete harmonic measure of boundary
arcs. The last result, Theorem~\ref{ThmDPwConv} devoted to the discrete Poisson kernel
$P^\mesh(\ccdot;v^\mesh;a^\mesh;\O^\mesh_\G)$ (see (\ref{DefDiscrPoisson})), needs more
technicalities, essentially because of the unboundedness of $P^\mesh$ near $a^\mesh$.

Let $g:\ol{B}(0,R)\to\R$ be a continuous function. Then, for a simply connected domain $\O\ss
B(0,R)$, let $h_g(\ccdot;\O):\O\to\R$ denote a unique solution of the Dirichlet boundary value
problem
\[
\Delta h_g(\ccdot;\O)=0\ \ \mathrm{inside}\ \O,\quad h_g(\ccdot;\O)=g\ \ \mathrm{on}\ \pa\O
\]
This is the classical result that the solution $h_g$ exists for any simply connected $\O$
(see, e.g., \S{}III.5, \S{}III.6 and Corollary~6.2 in \cite{GM05}). Note that this also
follows from the proof of Theorem~\ref{ThmDirConvUni}, where $h_g$ naturally appears as a
limit of discrete approximations.

Similarly, for a discrete simply-connected domain $\O^\mesh$, let
$H^\mesh_g=H^\mesh_g(\ccdot;\O^\mesh_\G)$ be a unique solution of the discrete Dirichlet
problem
\[
\Delta^\mesh H^\mesh_g=0\ \ \mathrm{in}\ \ \O^\mesh_\G,\qquad H^\mesh_g=g\ \ \mathrm{on}\ \
\pa\O^\mesh_\G.
\]

\begin{theorem}
\label{ThmDirConvUni} For any continuous $g:\ol{B}(0,R)\to\R$, the functions $H^\mesh_g$ are
uniformly \mbox{$C^1$\!-close} inside $\O^\mesh$ (in the sense of
Definition~\ref{DefUniformC1conv}) to $h^\mesh_g$. Moreover, the estimates
(\ref{HUniformClose}) and (\ref{HGradUnifromClose}) are also uniform in
\[
\textstyle g\in\cG_R(M,\n):=\{g:\max_{|z|\in \ol{B}(0,R)}|g(z)|\le M,\
\max_{|z-w|\le\r}|g(z)-g(w)|\le\n(\r)\},
\]
if both $M<+\infty$ and the modulus of continuity $\nu(\r)\to 0$ as $\r\to 0$ are fixed. In
other words, there exist $\ve(\mesh),\ \wt{\ve}(\mesh)\to 0$ as $\mesh\to 0$ (which may depend
on $r,R,M,\nu$) such that (\ref{HUniformClose}), (\ref{HGradUnifromClose}) are fulfilled for
any $g\in\cG_R(M,\nu)$ and any simply connected $\O\ss B(0,R)$.
\end{theorem}

\begin{proof}
Let $g$ be fixed. It is sufficient to verify both assumptions (a) and (b) in
Proposition~\ref{PropUniC1}. In fact, (a) was already essentially verified in the proof of
Proposition~\ref{PropDirConv}.

Indeed, $H^\mesh_g$ are uniformly bounded in $\O$ by a constant $M$, and so
Proposition~\ref{PropUniformComp} allows one to extract a convergent subsequence
$H^{\mesh_k}_g\rra H$. Thus, it is sufficient to prove that each subsequential limit $H$
coincides with $g$ on $\pa\O$.

Let $z=z^{\mesh_k}\in\O^{\mesh_k}_\G\ss\G^{\mesh_k}$, $w\in\pa\O$ be (one of) the closest to
$z$ points on $\pa\O$, and $d:=|z\!-\!w|$. Due to the geometric description of the kernel
convergence, there is a sequence of points $w^\mesh\in\pa\O^\mesh$ approximating $w$ as
$\mesh\to 0$. Thus, one still has $\dist(z;\pa\O^\mesh_\G)\le 2d$ for $\mesh$ small enough,
and the proof finishes exactly as before.

As it was pointed out in Remark~\ref{RemA=>B}, the Carath\'eodory stability of $h_g$ follows
from the same reasonings applied in the continuous setup. Namely, one can always find a
subsequence of the uniformly bounded harmonic functions $h_g(\ccdot;\O_k)$ uniformly
converging on compact subsets of $\O$ together with their gradients. Then, exactly as above,
the classical Beurling estimate implies that $h=g$ on $\pa\O$, and so each subsequential limit
coincides with $h_g(\ccdot;\O)$

Finally, for $g\in\cG_R(M,\n)$, let $\ve(\mesh;g)$ and $\wt{\ve}(\mesh;g)$ denote the best
possible bounds in (\ref{HUniformClose}) and (\ref{HGradUnifromClose}), respectively. Due to
the (both, discrete and continuous) maximum principles and the Harnack inequalities for
harmonic functions, one sees that
\[
|\ve(\mesh;g_1)-\ve(\mesh;g_2)|\le 2\|g_1\!-\!g_2\|_C\quad \mathrm{and}\quad
|\wt{\ve}(\mesh;g_1)-\wt{\ve}(\mesh;g_2)|\le\const\cdot \frac{\|g_1\!-\!g_2\|_C}{r},
\]
where $\|g\|_C:=\max_{z\in \ol{B}(0,R)}|g(z)|$ is the standard sup-norm in the space
$C(\ol{B}(0,R))$. Thus, $\ve(\mesh;\ccdot)$ and $\wt{\ve}(\mesh;\ccdot)$ are uniformly (in
$\mesh$) continuous (as functions of $g$) on the set $\cG_R(M,\n)$. Since
$\ve(\mesh;g),\wt{\ve}(\mesh;g)\to 0$ for any fixed $g\in\cG_R(M,\n)$, this implies
\[
\max_{g\in\cG_R(M,\n)}\ve_g(\mesh),\max_{g\in\cG_R(M,\n)}\wt{\ve}_g(\mesh)\to 0\ \
\mathrm{as}\ \ \mesh\to 0,
\]
due to the compactness of the set $\cG_R(M,\n)\ss C(\ol{B}(0,R))$.
\end{proof}

Let $\O^\mesh_\G$ be some bounded simply connected discrete domain. Recall that the discrete
Green's function $G_{\O^\mesh_\G}(\ccdot;v^\mesh)$, $v^\mesh\in\Int\O^\mesh_\G$ can be written
as
\[
G_{\Omega^\mesh_\G}^{\ }(\ccdot;v^\mesh)= G_\G(\ccdot;v^\mesh)-
G_{\Omega^\mesh_\G}^*(\ccdot;v^\mesh),
\]
where $G_{\Omega^\mesh_\G}^*=G_{\Omega^\mesh_\G}^*(\ccdot;v^\mesh):\O^\mesh_\G\to\R$ is a
solution of the discrete Dirichlet problem
\[
\D^\mesh G_{\Omega^\mesh_\G}^*=0\ \ \mathrm{in}\ \ \O^\mesh_\G,\qquad
G_{\Omega^\mesh_\G}^*=G_\G\ \ \mathrm{on}\ \ \pa\O^\mesh_\G.
\]
Theorem~\ref{ThmKenyonHarm} claims uniform \mbox{$C^1$-convergence} of the free Green's
function $G_\G$ to its continuous counterpart $G_\C(u;v):=\frac{1}{2\pi}\,\log|u\!-\!v|$ with
an error $O(\mesh^2|u\!-\!v|^{-2})$ for the functions and so $O(\mesh|u\!-\!v|^{-2})$ for the
gradients. Let $G^*_\O=G^*_\O(\ccdot;v):\O\to\R$ denote a solution of the corresponding
continous Dirichlet problem
\[
\D G^*_\O=0\ \ \mathrm{in}\ \ \O,\qquad G^*_\O=G_\C\ \ \mathrm{on}\ \ \pa\O.
\]

\begin{corollary}
\label{CorGreenConv} The discrete harmonic functions $G_{\Omega^\mesh_\G}^*(\ccdot;v^\mesh)$
are uniformly \mbox{$C^1$\!-close} inside $\O^\mesh$ (in the sense of
Definition~\ref{DefUniformC1conv}) to their continuous counterparts
$G_{\O^\mesh}^*(\ccdot;v^\mesh)$.
\end{corollary}

\begin{proof}
Let $g(u;v):=\frac{1}{2\pi}\max\{\log|u-v|,\log r\}$. Note that all the functions
$g(\ccdot;v)$ are uniformly bounded and equicontinuous. Let
$\D^\mesh\wt{G}^*_{\O^\mesh_\G}=\wt{G}^*_{\O^\mesh_\G}(\ccdot;v^\mesh)$ denote a solution of
the discrete Dirichlet problem
\[
\D^\mesh\wt{G}^*_{\O^\mesh_\G}=0\ \ \mathrm{in}\ \ \O^\mesh_\G,\qquad
\wt{G}^*_{\O^\mesh_\G}=g=G_\C\ \ \mathrm{on}\ \ \pa\O.
\]
Due to Theorem~\ref{ThmDirConvUni}, the functions $\wt{G}^*_{\O^\mesh_\G}$ are uniformly
\mbox{$C^1$\!-close} to $G_{\O^\mesh}^*$ inside $\O^\mesh$. On~the other hand, since
$B(v^\mesh,r)\ss\O^\mesh$, one has
\[
|G^*_{\O^\mesh_\G}-\wt{G}^*_{\O^\mesh_\G}|\le \const\cdot\mesh^2/r^2\ \mathrm{on}\
\pa\O^\mesh_\G.
\]
Then, the maximum principle and the discrete Harnack estimate (Corollary~\ref{CorLipHarm})
guarantees that $G^*_{\O^\mesh_\G}$ are uniformly \mbox{$C^1$\!-close} to
$\wt{G}^*_{\O^\mesh_\G}$ inside $\O^\mesh$.
\end{proof}

\begin{theorem}
\label{ThmDhmConv} The discrete harmonic measures $\dhm{\mesh}{\ccdot}{b^\mesh
a^\mesh}{\O^\mesh_\G}$ are uniformly \mbox{$C^1$\!-close} inside $\O^\mesh$ (in the sense of
Definition \ref{DefUniformC1conv}) to their continuous counterparts $\hm{\ccdot}{b^\mesh
a^\mesh}{\O^\mesh}$.
\end{theorem}

\begin{proof}
By conformal invariance, the continuous harmonic measure is Carath\'eodory stable, so the
second assumption in Proposition~\ref{PropUniC1} holds true. Thus, it is sufficient to prove
pointwise convergence (\ref{AssumptA_PWConv}) (see also Remark~\ref{RemA=>B}).

Let $(\O^\mesh;u^\mesh;a^\mesh,b^\mesh)\CaraTo (\O;u;a,b)$. The functions
$0\le\dhm{\mesh}{\ccdot}{b^\mesh a^\mesh}{\O^\mesh_\G}\le 1$ are uniformly bounded in $\O$.
Due to Proposition~\ref{PropUniformComp}, one can find a subsequence $\mesh_k\to 0$ such that
\[
\dhm{\mesh_k}{\ccdot}{b^{\mesh_k}a^{\mesh_k}}{\O^{\mesh_k}_\G}\rightrightarrows H
\]
uniformly on compact subsets of $\O$, where $H:\O\to\R$ is some harmonic function. It is
sufficient to prove that $H(u)=\hm{u}{ba}{\O}$ for each subsequential limit.

Let $z^\mesh\in\Int\O^\mesh_\G$. The weak Beurling-type estimate (see Proposition
\ref{PropWeakBeurling}) gives
\[
0\le \dhm{\mesh}{z^\mesh}{b^{\mesh}a^{\mesh}}{\O^{\mesh}_\G} \le \const \cdot
\lt[\frac{\dist(z^\mesh;\pa\O^\mesh_\G)}{\dist_{\O^\mesh_\G}(z^\mesh;b^{\mesh}a^{\mesh})}\rt]^\b
\]
uniformly as $\mesh\to 0$. Passing to the limit as $\mesh=\mesh_k\to 0$, one obtains
\[
0\le H(z) \le \const \cdot \lt[\frac{\dist(z;\pa\O)}{\dist_\O(z;ba)}\rt]^\b\quad \mathrm{for\
all}\ z\in\Int\O.
\]
Therefore, $H\equiv 0$ on the boundary arc $ab\ss P(\O)$. Similar arguments give $H\equiv 1$
on the arc $ba\ss P(\O)$. Hence, $H=\hm{\ccdot}{ba}{\O}$ and, in particular,
$H(u)=\hm{u}{ba}{\O}$.
\end{proof}

Let $\O^\mesh_\G$ be a simply connected discrete domain, $a^\mesh\!\in\pa\O^\mesh_\G$ and
$v^\mesh\!\in\Int\O^\mesh_\G$. We call \[
P^\mesh=P^\mesh(\ccdot;v^\mesh;a^\mesh;\O^\mesh_\G):\O^\mesh_\G\to\R
\]
the {\bf discrete Poisson kernel normalized at $\bm{v^\mesh}$}, if
\[
\D^\mesh P^\mesh=0\ \ \mathrm{in}\ \ \O^\mesh_\G,\quad P^\mesh=0\ \ \mathrm{on}\ \
\pa\O^\mesh_\G\setminus\{a^\mesh\},\quad \mathrm{and}\quad P^\mesh(v^\mesh)=1.
\]
Note that the function $P^\mesh$ is uniquely defined by these conditions (see
(\ref{DefDiscrPoisson})) and $P^\mesh\ge 0$.

In the continuous setup, let $\O$ be a simply connected domain, $a\in P(\O)$ be some prime end
and $v\in\Int\O$. Let $P=P(\ccdot;v;a;\O)$ denote a solution of the boundary value problem
\[
\D P=0\ \ \mathrm{in}\ \ \O,\quad P=0\ \ \mathrm{on}\ \ \pa\O\setminus\{a\},\quad P\ge 0,\quad
\mathrm{and}\quad P(v)=1
\]
(note that $P$ is uniquely defined by these conditions for any simply connected domain $\O$ as
the conformal image of the standard Poisson kernel defined in the unit disc ${\mathbb D}$).

\begin{theorem}
\label{ThmDPwConv} The discrete Poisson kernels $P^\mesh(\ccdot;v^\mesh;a^\mesh;\O^\mesh_\G)$
are uniformly \mbox{$C^1$\!-close} inside $\O^\mesh$ (in the sense of Definition
\ref{DefUniformC1conv}) to their continuous counterparts $P(\cdot;v;a^\mesh;\O^\mesh)$.
\end{theorem}

\begin{proof}
The continuous Poisson kernel $P(\ccdot;v,a,\O)$ is Carath\'eodory stable due to its conformal
invariant definition, so (\ref{AssumptB_CaraStab}) holds true. Thus, it is sufficient to prove
pointwise convergence (\ref{AssumptA_PWConv}) (see also Remark~\ref{RemA=>B}).

Let $(\O^\mesh;u^\mesh,v^\mesh;a^\mesh)\CaraTo (\O;u,v;a)$. Recall that $v^\mesh\to v$ and
$B(v,r)\ss\O^\mesh$ for some $r>0$, if $\mesh$ is small enough. It follows from $P(v^\mesh)=1$
and the discrete Harnack Lemma (Proposition~\ref{PropHarnack} (ii)) that $P^\mesh$ are
uniformly bounded on each compact subset of $\O$. Then, due to Proposition
\ref{PropUniformComp}, one can find a subsequence $\mesh_k\to 0$ such that
\[
P^{\mesh_k}(\ccdot;v^{\mesh_k};a^{\mesh_k};\O^{\mesh_k}_\G)\rightrightarrows H
\]
uniformly on compact subsets of $\O$, where $H\ge 0$ is some harmonic function in $\O$. It is
sufficient to prove that $H(u)=P(u;v;a;\O)$ for each subsequential limit $H$.

\begin{figure}
\centering{\mbox{\includegraphics[height=0.35\textheight]{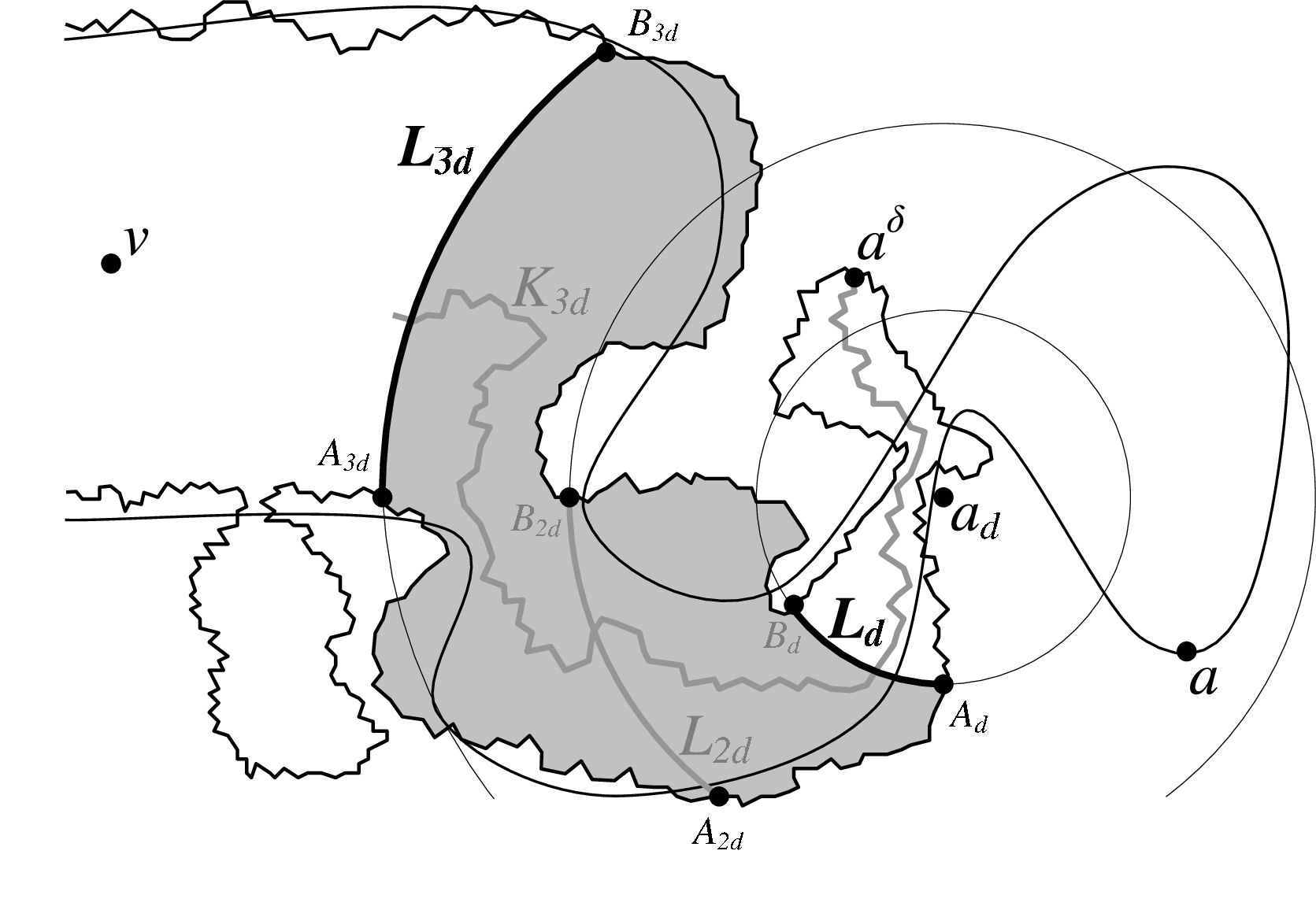}}}
\caption{\label{Fig:KLa} Parts of the continuous domain $\O$ and some discrete domain
$\O^\mesh$ close to $\O$ in the Carath\'eodory topology. Marked boundary points
\mbox{$a^\mesh\in\pa\O^\mesh$} and $a\in P(\O)$ are close to each other in this topology since
the corresponding small cross-cuts near $a_d$ are close. $L_d$ denotes the closest to $v$ arc
in $\{z:|z-a_d|=d\}\cap\O^\mesh$ which separates $v$ and $a^\mesh$. The~quadrilateral
$R_d^{3d}$ is shaded.}
\end{figure}

Let $d>0$ be small enough. Then, there exists a crosscut $\g^a_d\ss B(a_d,\frac{1}{2}d)$ in
$\O$ separating $a$ and $u,v$ (see Fig.~\ref{Fig:KLa}). Moreover, one may assume that
$u,v\notin B(a_d,4d)$ and $u,v$ belong to the same component of $\O\setminus\ol{B(a_d,4d)}$.
For sufficiently small $\mesh$ let
\[
L_d\ss \{z:|z\!-\!a_d|=d\}\cap\O^\mesh\]
be an arc separating $v^\mesh$ and $a^\mesh$ in $\O^\mesh$
(we take the arc closest to $v^\mesh$, see Fig.~\ref{Fig:KLa}). Let $\O^\mesh_d$ denote the
connected component of $\O^\mesh\setminus L_d$ containing $v^\mesh$. Since $\O^\mesh\CaraTo\O$
w.r.t. $v$ and $v^\mesh\to v$, one has
\begin{equation}
\label{xHmLd} \hm{v^\mesh}{L_d}{\O^\mesh_d}\ge \const(d)>0
\end{equation}
(here and below constants $\const(d)$ do not depend on $\mesh$). Similarly, let
$\O^\mesh_{3d}$ be the connected component of $\O^\mesh\setminus L_{3d}$ containing $u,v$.
Denote
\[
M^\mesh_{3d}=\max\{P^\mesh(z^\mesh),\ z^\mesh\in \O^\mesh_{3d}\cap\G^\mesh\}.
\]
Since the function $P^\mesh$ is discrete harmonic, one has
\[
M^\mesh_{3d}= P^\mesh(z^\mesh_0)\le P^\mesh(z^\mesh_1)\le P^\mesh(z^\mesh_2)\le..
\]
for some nearest-neighbor path $K_{3d}^\mesh=\{z^\mesh_0\sim z^\mesh_1\sim z^\mesh_2\sim ..,\
z^\mesh_s\in\O^\mesh_\G\}$, starting at some $z^\mesh_0\in \O^\mesh_{3d}$. Since
$P^\mesh|_{\pa\O^\mesh_\G\setminus\{a^\mesh\}}=0$, the unique possibility for this path to end
is $a^\mesh$.

Using (\ref{xHmLd}), it is not hard to conclude (see Lemma~\ref{xLemmaHm} below) that the
following holds true for the {\it continuous} harmonic measures:
\[
\hm{v^\mesh}{K_{3d}}{\O^\mesh\setminus K_{3d}}\ge \hm
{v^\mesh}{K_{3d}\cap\O^\mesh_d}{\O^\mesh_d\setminus K_{3d}} \ge
\const\cdot\,\hm{v^\mesh}{L_d}{\O^\mesh_d}\ge \const(d),
\]
where $K_{3d}$ is the corresponding polyline starting at $z_0^\mesh$ and ending at $a^\mesh$.
Applying Theorem~\ref{ThmDhmConv} with $\ve=\frac{1}{2}\const(d)$, one obtains the same
inequality
\[
\dhm{\mesh}{v^\mesh}{K^\mesh_{3d}}{\O^\mesh_\G\setminus K^\mesh_{3d}}\ge \const(d)>0
\]
for {\it discrete} harmonic measures {\it uniformly} as $\mesh\to 0$ (with smaller
$\const(d)$). Recall that $P^\mesh(v^\mesh)=1$ by definition and $P^\mesh(v)\ge M^\mesh_{3d}$
along the path $M^\mesh_{3d}$. Thus,
\[
M^\mesh_{3d}\le \const(d),\quad \mathrm{if}\ \mesh\ \mathrm{is\ small\ enough}.
\]
Finally, let $z^\mesh\in \G^\mesh\cap\O^\mesh_{3d}$ be such that $|z^\mesh\!-\!a_d|>3d$. The
weak Beurling-type estimate immediately gives
\[
P^\mesh(z^\mesh)\le \const \cdot
\lt[\frac{\dist(z^\mesh;\pa\O^\mesh)}{\dist(z^\mesh;L_{3d})}\rt]^\b \cdot M^\mesh_{3d}\le
\const(d)\cdot \lt[\frac{\dist(z^\mesh;\pa\O^\mesh)}{|z^\mesh\!-\!a_d|-3d}\rt]^\b.
\]
Passing to the limit as $\mesh=\mesh_k\to 0$, one obtains
\[
H(z) \le \const(d)\cdot \lt[\frac{\dist(z;\pa\O)}{|z\!-\!a_d|-3d}\rt]^\b\quad \mathrm{for\
all}\ z\in\O_{3d}\ \ \mathrm{such\ that}\ \ |z\!-\!a_d|>3d.
\]
Thus, $H\!=\!0$ on $\pa\O\setminus \{a\}$. Since $H\ge 0$ and $H(v)\!=\!1$, this gives
$H\!=\!P(\ccdot;v;a;\O)$.
\end{proof}

\begin{lemma}
\label{xLemmaHm} Let $\O\ss\C$ be some simply connected domain, $v\in\Int\O$ and $a\in P(\O)$.
Let $L_d\ss\{z:|z\!-\!a_d|=d\}\cap\O$ be the arc separating $v$ and $a$ that is
 closest to $v$, and
$\O_d$ be the connected component of $\O\setminus L_d$ containing $v$. Let $K_{3d}$ be some
path connecting $L_{3d}$ and $L_d$ inside the conformal quadrilateral
$R_d^{3d}=\O_d\setminus\ol{\O_{3d}}$ (see Fig.~\ref{Fig:KLa}). Then
\[
\hm{v}{K_{3d}}{\O_d\setminus K_{3d}} \ge \const\cdot\,\hm{v}{L_d}{\O_d}
\]
for some absolute positive constant.
\end{lemma}

\begin{proof}
Note that $\hm{v}{L_d}{\O_d}\le \hm{v}{L_d}{\O_d\setminus K_{3d}} +
\hm{v}{K_{3d}}{\O_d\setminus K_{3d}}$. Thus, it is sufficient to prove that
\[
\hm{v}{L_d}{\O_d\setminus K_{3d}} \le \const\cdot\,\hm{v}{K_{3d}}{\O_d\setminus K_{3d}}.
\]
Furthermore, monotonicity arguments give
\[
\hm{v}{L_d}{\O_d\setminus K_{3d}} \le \int_{L_{2d}} \hm{v}{|dz|}{\O_{2d}\setminus
K_{3d}}\cdot\,\hm{z}{L_d\cup L_{3d}}{R_d^{3d}\setminus K_{3d}}
\]
and, in a similar manner,
\[
\hm{v}{K_{3d}}{\O_d\setminus K_{3d}} \ge \int_{L_{2d}} \hm{v}{|dz|}{\O_{2d}\setminus
K_{3d}}\cdot\,\hm{z}{K_{3d}}{R_d^{3d}\setminus K_{3d}}.
\]
Let $L_d=A_d B_d$ and so on (see Fig.~\ref{Fig:KLa}). Applying monotonicity arguments once
more, one sees
\[
\hm{z}{K_{3d}}{R_d^{3d}\setminus K_{3d}}\ge \min\left\{ \hm{z}{A_{3d}A_d}{R_d^{3d}}\,,\,
\hm{z}{B_dB_{3d}}{R_d^{3d}}\right\}.
\]
Thus, it is sufficient to prove that
\[
\hm{z}{A_dB_d\cup B_{3d}A_{3d}}{R_d^{3d}}\le \const\cdot\min\left\{
\hm{z}{A_{3d}A_d}{R_d^{3d}}\,,\, \hm{z}{B_dB_{3d}}{R_d^{3d}}\right\}
\]
for all $z\in L_{2d}$. Due to the conformal invariance of harmonic measure, the last estimate
follows from the uniform bounds on the extremal distances (conformal modulii of
quadrilaterals)
\[
\lambda_{A_{nd}B_{nd}B_{md}A_{md}}(A_{nd} A_{md}\,;\,B_{md} B_{nd})\ge
\frac{1}{2\pi}\,\log\frac{m}{n}\,,\qquad 1\le n< m\le 3.\qedhere
\]
\end{proof}

\subsection{Boundary Harnack principle and normalization on a ``straight'' part of the boundary}
\label{SectBoundaryNorm}

Recall that $\H^\mesh$ denotes the polygonal representation of a half-plane $\H=\{z:\Im z>0\}$
discretization (i.e., the union of all faces, edges and vertices that intersect $\H$, see Fig.
\ref{Fig:DiscrHalfPlane}B). As for bounded domains, denote by
$\dhm\mesh{u^\mesh}{\{x^\mesh\}}{\H^\mesh_\G}$ the probability of the event that the random
walk starting at \mbox{$u^\mesh\in\H^\mesh_\G$} first hits the boundary $\pa\H^\mesh_\G$ at a
vertex \mbox{$x^\mesh\in\pa\H^\mesh_\G$}. It is easy to see (e.g., using the unboundedness of
the free Green's function (\ref{Gasympt}) or Proposition \ref{PropWeakBeurling}) that
\[
\sum_{x^\mesh\in\pa\H^\mesh_\G}\dhm\mesh{u^\mesh}{\{x^\mesh\}}{\H^\mesh_\G}=1.
\]
Let
\begin{equation}
\label{gImDef} \gIm^\mesh(u^\mesh)= \Im u^\mesh - \sum_{x^\mesh\in\pa\H^\mesh_\G}\Im
x^\mesh\cdot \dhm\mesh{u^\mesh}{\{x^\mesh\}}{\C_+^\mesh}\quad \mathrm{for}\ \ u^\mesh\in
\H^\mesh_\G\,.
\end{equation}
The function $\gIm^\mesh$ is discrete harmonic in $\H^\mesh_\G$, $\gIm^\mesh=0$ on $\pa
\H^\mesh_\G$ and $|\gIm^\mesh(u^\mesh)-\Im u^\mesh|\le 2\mesh$ for all $u^\mesh\in\H^\mesh_\G$
(note that these conditions define $\gIm^\mesh$ uniquely). In particular, if $\Im
u^\mesh\in[3\mesh,5\mesh]$, then $\gIm^\mesh(u^\mesh)\asymp\mesh$ (here and below we write
\[
f\asymp g\quad \mathrm{iff}\ \ \const_1\cdot f \le g \le \const_2\cdot g
\]
for some positive absolute constants). Since $\gIm^\mesh\ge 0$ is discrete harmonic, this
implies
\[
\gIm^\mesh(x^\mesh_{\mathrm{int}})\asymp\mesh\ \ \mathrm{for\ all}\
x^\mesh=(x^\mesh;(x^\mesh_{\mathrm{int}}x^\mesh))\in\pa\H^\mesh_\G.
\]
Below we say that a discrete domain $\O^\mesh$ has a {\bf``straight'' boundary} near
$x^\mesh\in\pa\O^\mesh$, if $\O^\mesh$ and $\H^\mesh$ coincide near $x^\mesh$ (certainly, it's
more natural to include not only $\H^\mesh$ itself but all discrete half-planes into the
definition but $\H^\mesh$ will be sufficient for our purposes).
\begin{definition} \label{DefDnH}
For a function $H$ defined in a domain $\O^\mesh$ having a ``straight'' boundary near
$x^\mesh$ we define the value of its (inner) normal derivative at $x^\mesh$ as
\begin{equation}
\label{DefDnHformula} [\pa_n^\mesh H](x^\mesh):=
\frac{H(x^\mesh_{\mathrm{int}})-H(x^\mesh)}{\gIm^\mesh(x^\mesh_{\mathrm{int}})}.
\end{equation}
\end{definition}
\begin{remark}\label{RemDnH} In other words, we use the value
$\gIm^\mesh(x^\mesh_{\mathrm{int}})$ as a natural normalization constant, so that
$[\pa_n^\mesh \gIm^\mesh](x^\mesh)=1$. Note that, if $H(x^\mesh)=0$, then $[\pa_n^\mesh
H](x^\mesh)\asymp \mesh^{-1}H(x^\mesh_{\mathrm{int}})$.
\end{remark}
Below we need some rough estimates for the discrete harmonic measure in rectangles. Let
$R(s,t):=(-s;s)\ts(0;t)\ss\C$ be an open rectangle, $o^\mesh\in\pa\H^\mesh_\G$ denote the
closest to $0$ boundary vertex, $R^\mesh_\G=R^\mesh_\G(s,t)\ss\G$ be the discretization of
$R(s,t)$, and
\begin{align*}
L^\mesh_\G=L^\mesh_\G(s) := & \left\{v^\mesh\in \pa R^\mesh_\G(s,t): \Im v^\mesh \le
0\right\}, \cr U^\mesh_\G=U^\mesh_\G(s,t) := & \left\{v^\mesh\in \pa R^\mesh_\G(s,t): \Im
v^\mesh \ge t\right\}, \cr V^\mesh_\G=V^\mesh_\G(s,t) := & \left\{v^\mesh\in \pa
R^\mesh_\G(s,t): |\Re v^\mesh| \ge s\right\}
\end{align*}
be the lower, upper and vertical parts of the boundary $\pa R^\mesh_\G(s,t)$ (see Fig.
\ref{Fig:DiscrHalfPlane}B).
\begin{lemma}
\label{LemmaDhmR} Let $t\ge 2\mesh$ and $s\ge 2t$. Then
\[
\dhm\mesh{o^\mesh_{\mathrm{int}}}{U^\mesh_\G}{R^\mesh_\G(s,t)}\asymp \mesh/t\quad
\mathit{and}\quad \dhm\mesh{o^\mesh_{\mathrm{int}}}{V^\mesh_\G}{R^\mesh_\G(s,t)}\le
\const\cdot\,\mesh t/s^2.
\]
\end{lemma}
\begin{remark} The last estimate is very rough but sufficient for us. Standard arguments
similar to the proof of Proposition~\ref{PropWeakBeurling} easily give an exponential bound.
\end{remark}
\begin{proof}
We consider two harmonic polynomials
\[
h_1(x+iy)=\frac{y\!+\!2\mesh}{t\!+\!2\mesh}\qquad\mathrm{and}\qquad
h_2(x+iy):=\frac{y}{t\!+\!2\mesh}\,-\,\frac{x^2+(y\!+\!2\mesh)(t\!+\!2\mesh\!-\!y)}{s^2}.
\]
Their restrictions on $\G$ are discrete harmonic due to Lemma \ref{ApproxLemma}~(i), and
\begin{align*}
h_1(x+iy)\ge 1\ge h_2(x+iy),&\quad \mathrm{if}\ \ y\in[t;t\!+\!2\mesh], \cr h_1(x+iy)\ge 0\ge
h_2(x+iy),&\quad \mathrm{if}\ \ y\in[-2\mesh;0],\cr h_1(x+iy)\ge 0\ge h_2(x+iy),&\quad
\mathrm{if}\ \ y\in [-2\mesh,t\!+\!2\mesh]\ \ \mathrm{and}\ \ |x|\in [s;s\!+\!2\mesh].
\end{align*}
Thus, $h_1(v^\mesh)\ge \dhm\mesh{v^\mesh}{U^\mesh_\G}{R^\mesh_\G}\ge h_2(v^\mesh)$ for all
$v^\mesh\in\pa R^\mesh_\G$, and so, by the maximum principle, for all $v^\mesh\in R^\mesh_\G$.
In particular, if $t\ge 5\mesh$ (the case $t\le 5\mesh$ is trivial), then
\[
\frac{7\mesh}{t\!+\!2\mesh}\ge\dhm\mesh{v^\mesh}{U^\mesh_\G}{R^\mesh_\G}\ge
\frac{3\mesh}{t\!+\!2\mesh}-\frac{5\mesh t}{s^2}\ge\frac{\mesh}{t\!+\!2\mesh}\quad
\mathrm{for}\ \ v^\mesh\in [-2\mesh;2\mesh]\times [3\mesh;5\mesh],
\]
because of $5t(t\!+\!2\mesh)\le 7t^2\le 2s^2$.  Since $\o^\mesh$ is discrete harmonic and
nonegative, we obtain $\o^\mesh\asymp \mesh/t$ everywhere near $o^\mesh$. The upper bound for
$\dhm\mesh{o^\mesh}{V^\mesh_\G}{R^\mesh_\G}$ follows by the consideration of the quadratic
harmonic polynomial
\[
h_3(x+iy)=\frac{x^2+(y\!+\!2\mesh)(t\!+\!2\mesh\!-\!y)}{s^2}
\]
which is nonnegative on $L^\mesh_\G\cup U^\mesh_\G$ and not less than $1$ on $V^\mesh_\G$.
\end{proof}

\begin{proposition}[\bf Boundary Harnack principle]
\label{PropBoundHar} Let $t \ge \mesh$, $H$ be a nonnegative discrete harmonic function in a
discrete rectangle $R^\mesh(2t,2t)$, $o^\mesh$ be the boundary vertex closest to $0$, and
$c^\mesh$ denote the inner vertex closest to the point $c=it$. If $H=0$ everywhere on the
lower boundary $L^\mesh(2t)$, then the double-side estimate
\[
[\pa^\mesh_n H](o^\mesh)\asymp \frac{H(c^\mesh)}{t}
\]
holds true with some constants independent of $\mesh$ and $t$.
\end{proposition}

\begin{proof} Recall that $[\pa^\mesh_n H](o^\mesh)\asymp \mesh^{-1}\cdot
H(o^\mesh_{\mathrm{int}})$ (see Remark~\ref{RemDnH}). Let $t\ge 4\mesh$ (the case $t\le
4\mesh$ is trivial). It follows from discrete Harnack Lemma
(Proposition~\ref{PropHarnack}~(ii)) that the values of $H$ on $U^\mesh_\G(t,\frac{1}{2}t)$
are uniformly comparable with $H(c^\mesh)$. Then,
\[
H(o^\mesh_{\mathrm{int}}) \ge \dhm\mesh{o^\mesh_{\mathrm{int}}}{U^\mesh_\G}
{R^\mesh_\G(t,{\textstyle\frac{1}{2}}t)}\cdot \const \cdot H(c^\mesh) \ge \const
\cdot{\mesh}/{t}\cdot H(c^\mesh).
\]
Further, note that $M:=\max_{v\in R^\mesh_\G(t,\frac{1}{2}t)} H(v) \le \const\cdot
H(c^\mesh)$. Indeed, by the maximum principle, $H\ge M$ holds true along some nearest-neighbor
path $K$ running from $\pa R^\mesh(t,{\textstyle\frac{1}{2}}t)$ to $U^\mesh(2t,2t)$ or
$V^\mesh(2t,2t)$ (this path cannot end on $L^\mesh(2t)$ since $H=0$ there). Arguing as in the
proof of Proposition~\ref{PropWeakBeurling}, it is easy to see that the probability that the
random walk starting at $c^\mesh$ hits $K$ before $\pa R^\mesh(2t,2t)$ is bounded below by
some absolute constant, so $H(c^\mesh)\ge\const\cdot M$. Then, Lemma~\ref{LemmaDhmR} gives
\[
H(o^\mesh_{\mathrm{int}}) \le \dhm\mesh{o^\mesh_{\mathrm{int}}}{U^\mesh_\G\cup V^\mesh_\G}
{R^\mesh_\G(t,{\textstyle\frac{1}{2}}t)}\cdot M \le \const \cdot{\mesh}/{t}\cdot
H(c^\mesh).\qedhere
\]
\end{proof}

From now on, we consider only discrete domains $(\O^\mesh_\G;a^\mesh)$ such that
\begin{equation}
\label{xAssumptionOdisc} R^\mesh_\G(T,T)\ss\O^\mesh_\G,\qquad
L^\mesh_\G(T)\ss\pa\O^\mesh_\G,\quad \mathrm{and}\quad a^\mesh\in\pa\O^\mesh_\G\setminus
L^\mesh_\G(T)
\end{equation}
for some $T>0$. Note that all continuous domains $(\O;a)$ appearing as Carath\'eodory limits
of these $(\O^\mesh;a^\mesh)$ satisfy
\begin{equation}
\label{xAssumptionOcont} (-T;T)\!\times\!(0;T)\ss\O,\qquad [-T;T]\ss\pa\O\quad
\mathrm{and}\quad a\in\pa\O\setminus (-T;T).
\end{equation}

For a domain $(\O;a)$ satisfying (\ref{xAssumptionOcont}), we define the {continuous Poisson
kernel \mbox{$P_0=P_0(\ccdot;a;\O)$} normalized at $0$} as the unique solution of the boundary
value problem
\[
\D P_0=0\ \ \mathrm{in}\ \ \O,\quad P_0=0\ \ \mathrm{on}\ \ \pa\O\setminus\{a\},\quad P_0\ge
0,\quad \mathrm{and}\quad [\pa_n P_0](0)=1.
\]
where $[\pa_nP_0](0)=[\pa_yP_0](0)$ denotes the (inner) normal derivative of $P_0$ at $0$.

For a discrete domain $(\O^\mesh_\G;a^\mesh)$ satisfying (\ref{xAssumptionOdisc}), we call
$P^\mesh_{o^\mesh}=P^\mesh_{o^\mesh}(\ccdot;a^\mesh;\O^\mesh_\G)$ the {\bf discrete Poisson
kernel normalized at $\bm{o^\mesh}$}, if
\[
\D^\mesh P^\mesh_{o^\mesh}=0\ \ \mathrm{in}\ \ \O^\mesh_\G,\quad P^\mesh_{o^\mesh}=0\ \
\mathrm{on}\ \ \pa\O^\mesh_\G\setminus\{a\},\quad \mathrm{and}\quad [\pa_n^\mesh
P^\mesh_{o^\mesh}](o^\mesh)=1,
\]
where the discrete normal derivative $\pa^\mesh_n$ is given by (\ref{DefDnHformula}). Note
that $P^\mesh_{o^\mesh}$ is uniquely defined by these conditions, namely
\[
P^\mesh_{o^\mesh}(\ccdot;a^\mesh;\O^\mesh_\G)= \dhm\mesh\ccdot{a^\mesh}{\O^\mesh_\G}\cdot
\frac{\gIm^\mesh(o^\mesh_{\mathrm{int}})}
{\dhm\mesh{o^\mesh_{\mathrm{int}}}{a^\mesh}{\O^\mesh_\G}}\,.
\]

\begin{theorem}
\label{ThmDPoConv} The discrete Poisson kernels
$P^\mesh_{o^\mesh}(\ccdot;a^\mesh;\O^\mesh_\G)$ defined for the class of discrete domains
satisfying (\ref{xAssumptionOdisc}) with some $T>0$ are uniformly \mbox{$C^1$\!-close} inside
$\O^\mesh$ (in the sense of Definition \ref{DefUniformC1conv}) to the continuous Poisson
kernels $P_0(\ccdot;a^\mesh;\wt{\O}^\mesh)$, where $\wt{\O}^\mesh$ denotes the modified
polygonal representation of the discrete domain $\O^\mesh_\G$ with the ``straight'' part of
the boundary $L^\mesh(T)\ss\pa\O^\mesh$ replaced by the straight segment $[-T,T]$. The rate of
the uniform convergence may depend on $T$.
\end{theorem}

\begin{proof}
The continuous Poisson kernel $P_0$ is Carath\'eodory stable, so (\ref{AssumptB_CaraStab})
holds true. Thus, it is sufficient to check (\ref{AssumptA_PWConv}).

Let $(\O^\mesh;u^\mesh;a^\mesh)\CaraTo (\O;u;a)$ and $c^\mesh$ denote the vertex closest to
the point $\frac{1}{2}iT$. Due to the boundary Harnack principle
(Proposition~\ref{PropBoundHar}), the values $P^\mesh_{o^\mesh}(c^\mesh)$ are uniformly
bounded by some constant (depending on $T$). Hence,
$P^\mesh_{o^\mesh}(\ccdot;a^\mesh;\O^\mesh_\G)$ are uniformly (w.r.t~$\mesh$) bounded on each
compact subset of $\O$ because of the discrete Harnack Lemma (Proposition~\ref{PropHarnack}).
Then, due to Proposition~\ref{PropUniformComp}, one can take a subsequence $\mesh_k\to 0$ so
that
\[
P^{\mesh_k}_{o^\mesh}(\ccdot;a^{\mesh_k};\O^{\mesh_k}_\G)\rightrightarrows H
\]
uniformly on compact subsets of $\O$, where $H\ge 0$ is some harmonic in $\O$ function. We
need to prove that $H(u)=P_0(u;a;\O)$ for each subsequential limit $H$.

Repeating the arguments given in the proof of Theorem~\ref{ThmDPwConv}, one obtains that,
first, for each $r>0$ the functions $P^{\mesh}_{o^\mesh}$ are uniformly bounded {\it
everywhere} in $\O^\mesh_\G$ away from $a^\mesh$ (in particular, everywhere in the smaller
rectangle $R^\mesh_\G(\frac{1}{2}T,\frac{1}{2}T)$) and, second, $H=0$ on
$\pa\O\setminus\{a\}$. Therefore, due to $H\ge 0$,
\[
H= \mu P_0(\ccdot;a;\O)\quad \mathrm{for\ some}\ \ \m\ge 0.
\]

Now one needs to prove that $\mu=1$. Let
\[
Q^\mesh(\ccdot):=P^\mesh_{o^\mesh}(\ccdot;a^\mesh;\O^\mesh_\G)-\gIm^\mesh(\ccdot).
\]
By definition, the function $Q^\mesh$ is discrete harmonic in $R^\mesh_\G(T,T)$, $Q^\mesh=0$
on the lower boundary $L^\mesh_\G(T)$, $Q^\mesh(o^\mesh_{\mathrm{int}})=0$ and
\[
Q^\mesh(v)\rightrightarrows \mu P_0(v;a;\O) - \Im v\ \ \mathrm{as}\ \ \mesh=\mesh_k\to 0
\]
uniformly on compact subsets of $R(T,T)$. Since $P_0|_{(-T,T)}=0$ and $[\pa_nP_0](0)=1$, one
has
\[
P_0(x\!+\!iy)=y+O(xy+y^2)\ \ \mathrm{for}\ \ \textstyle x+iy\in
[-\frac{1}{2}T;\frac{1}{2}T]\times[0;\frac{1}{2}T].
\]
Thus, for any fixed $T\gg s\gg t>0$, the following hold true:
\[
Q^\mesh \rra (\m\!-\!1)t + O(st)\ \ \mathrm{as}\ \ \mesh\to 0\ \ \mathrm{uniformly\ \ on}\ \
U^\mesh_\G(t,s),
\]
\[
Q^\mesh=0\ \ \mathrm{on}\ \ {L^\mesh_\G(s)} \qquad\mathrm{and}\qquad |Q^\mesh|\le\const\ \
\mathrm{on}\ \ V^\mesh_\G(s,t).
\]
Then, the normalization $Q^\mesh(o^\mesh_{\mathrm{int}})=0$ and Lemma~\ref{LemmaDhmR} give
\[
\left|(\m\!-\!1)t + O(st)+o_{\mesh\to 0}(1)\right|\cdot\mesh/t \le \const\cdot\,\mesh t/s^2\ \
\mathrm{as}\ \ \mesh\to 0.
\]
So, for any $s$ and $t$, one has $|\mu-1| \le\const\cdot (s+t/s^2)$. Setting $t:=s^3$ and
passing to the limit as $s\to 0$, one arrives at $\m\!=\!1$.
\end{proof}

\renewcommand{\thesection}{A}
\section{Appendix}
\setcounter{equation}{0}
\renewcommand{\theequation}{A.\arabic{equation}}

\subsection{Kenyon's asymptotics for the Green's function and the Cauchy kernel}

\label{SectAKenyonFreeGreen}

Below we give a sketch of R.~Kenyon's \cite{Ken02} arguments. See also \cite{Buck08}.

\begin{proof}[\bf Proof of the Theorem~\ref{ThmKenyonHarm}]
Following J.~Ferrand \cite{Fer44}, R.~Kenyon \cite{Ken02} and Ch.~Mercat \cite{Mer07}, we
introduce {\bf discrete exponentials}
\begin{equation}
\label{DiscrExp} e(\l,u;u_0):= \prod_{j=0}^{2k-1} \frac{1+\frac{\l}{2}(u_k\!-\!u_{k-1})}
{1-\frac{\l}{2}(u_k\!-\!u_{k-1})},
\end{equation}
where $P_{u_0u}=u_0u_1u_2...u_{2k-1}u_{2k}$ is a path from $u_0$ to $u_{2k}=u$ on the
corresponding rhombic lattice (thus, $u_{2j}\in\G$ and $u_{2j-1}\in\G^*$ for all $j$). We
prefer the parametrization which is closest to the continuous case, so that $e(\l,u;u_0)\to
\exp[\l(u\!-\!u_0)]$ as $\mesh\to 0$. It's easy to see that this definition does not depend on
the choice of the path. Since the angles of rhombi are bounded from $0$ and $\pi$, one can
choose $P_{u_0u}$ so that the following condition holds:
\begin{quotation}
\noindent for all $j$ either (a) $|\arg(u_{j+1}\!-\!u_j)-\arg(u\!-\!u_0)|< \frac{\pi}{2}$ or
(b) $u_j$ and $u_{j+2}$ are opposite vertices of some rhombus and
\mbox{$|\arg(u_{j+2}\!-\!u_j)-\arg(u\!-\!u_0)|< \frac{\pi}{2}$}. In particular,
$\arg(u_{j+1}\!-\!u_j)-\arg(u\!-\!u_0)\in(-\pi,\pi)$ for all $j$.
\end{quotation}
Define (see R.~Kenyon \cite{Ken02} and A.~Bobenko, Ch.~Mercat and Yu.~Suris \cite{BMS05})
\begin{equation}
\label{AgDefInt} \wt{G}_\G(u;u_0) = \frac{1}{8\pi^2i}\int_C \frac{\log
\l}{\l}\,e(\l,u;u_0)\,d\l.
\end{equation}
where $C$ is a curve  which runs counter clockwise around the disc of (large) radius $R$ from
the angle $\arg\ol{(u\!-\!u_0)}-\pi$ to $\arg\ol{(u\!-\!u_0)}+\pi$, then along the segment
$e^{i\arg\ol{(u-u_0)}}[-R,-r]$, then clockwise around the disc of (small) radius $r$ and then
back along the same segment (the integral does not depend on the $\log$ branch, since
$e(0,u;u_0)=e(\iy,u;u_0)=1$).

This function is discrete harmonic away from $u_0$ since all discrete exponentials are
harmonic (as functions of $u$) and one can use the same contour of integration for all
$u_s\sim u$. Furthermore, $\wt{G}(u_0;u_0)=0$ and, by straightforward computation,
\[
\wt{G}_\G(u_s;u_0)=\frac{\theta_s\cot\theta_s}{\pi},\ \ \mathrm{if}\ \ u_s\sim u,\qquad
\mathrm{so}\qquad \D^\mesh \wt{G}_\G(u_0;u_0)\cdot\weightG{u_0}=1.
\]

Rotating and scaling the plane, one may assume that $\arg(u\!-\!u_0)=0$ and $\mesh=1$. It's
easy to see that the contribution of intermediate $\l=-t<0$ to the integral in
(\ref{AgDefInt}) is exponentially small. Indeed, in case (a) one has
\[
\lt|\frac{1-\frac{t}{2}e^{i\b_j}} {1+\frac{t}{2}e^{i\b_j}}\rt|^2 = 1 -
\frac{8t\cos\b_j}{t^2+4t\cos\b_j+4}\le \exp\lt[-\frac{8t\cos\b_j}{(t+2)^2}\rt],
\]
where $\b_j=\arg(u_{j+1}\!-\!u_j)$ and $\cos\b_j=\Re(u_{j+1}\!-\!u_j)>0$. Similarly, in case
(b),
\[
\lt|\frac{(1-\frac{t}{2}e^{i\b_j})(1-\frac{t}{2}e^{i\b_{j+1}})}
{(1+\frac{t}{2}e^{i\b_j})(1+\frac{t}{2}e^{i\b_{j+1}})}\rt|^2
\le\exp\lt[-\frac{8t(\cos\b_j\!+\!\cos\b_{j+1})}{(t+2)^2}\rt],
\]
due to $\cos\b_j\!+\!\cos\b_{j+1}=\Re(u_{j+2}\!-\!u_j)>0$. Thus,
\[
|e(-t,u;u_0)|\le \exp\lt[-\frac{4t(u\!-\!u_0)}{(t+2)^2}\rt]
\]
and the asymptotics of (\ref{AgDefInt}) as $|u\!-\!u_0|\to\iy$ are determined by the
asymptotics of $e(\l,u;u_0)$ near $0$ and $\iy$. Some version of the Laplace method (see
\cite{Ken02} and \cite{Buck08}) gives
\[
\wt{G}_\G(u;u_0)=\frac{1}{2\pi}\log|u\!-\!u_0| + \frac{\gamma_{\mathrm{Euler}}+\log
2}{2\pi}+O(|u\!-\!u_0|^{-2}),
\]
where the remainder is of order $|u\!-\!u_0|^{-2}$ due to
\[
\frac{d^2}{d\l^2}\,\log (e(\l,u;u_0))\Big|_{\l=0}=\frac{d^2}{d\l^2}\,\log
(e(\l,u;u_0))\Big|_{\l=\iy}=0.
\]
The uniqueness of $G(\ccdot;u_0)$ (and $\Im G=0$) easily follows by the Harnack inequality
(Corollary~\ref{CorLipHarm}). Indeed, $G:=G_1(\ccdot;u_0)-G_2(\ccdot;u_0)$ would be discrete
harmonic everywhere on $\G$ and $\max_{|u|\le R} |H(u)|/R\to 0$ as $R\to\iy$, so $G(u)\equiv
G(u_0)=0$.
\end{proof}

\begin{proof}[\bf Proof of Theorem~\ref{ThmKenyonHol}]
As in Theorem~\ref{ThmKenyonHarm}, $K(\ccdot;z_0)$ can be explicitly constructed using
(modified) discrete exponentials. Similarly to (\ref{DiscrExp}), denote
\begin{equation}
\label{xEv0z0Def} e(\l,u_0^{\pm};z_0) :=
\frac{1}{(1-\frac{\l}{2}(u_0^{\pm}\!-\!w_0^{-}))(1-\frac{\l}{2}(u_0^{\pm}\!-\!w_0^{+}))}
\end{equation}
for the ``black'' vertices $u_0^{\pm}\in\G$ of the rhombus $u_0^{-}w_0^{-}u_0^{+}w_0^{+}$
centered at $z_0$ and, by induction,
\[
\frac{e(\l,w;z_0)}{e(\l,u;z_0)}:=\frac{1+\frac{\l}{2}(w\!-\!u)}{1-\frac{\l}{2}(w\!-\!u)}\quad
\mathrm{for\ all}\ \ w\sim u,\ u\in\G,\ w\in\G^*.
\]
Then, all $e(\l,{\ccdot};z_0)$ are well-defined and discrete holomorphic on $\L$. Let (see
\cite{Ken02})
\[
K(v;z_0)=\frac{1}{\pi}\int_{-\ol{(v-z_0)}\iy}^0
e(\l,v;z_0)d\l,
\]
where the integral being, say, along the ray $\arg\z=\arg\ol{(v\!-\!z_0)}\pm\pi$ (taking the
path from $z_0$ to $u$ as in the proof of Theorem~\ref{ThmKenyonHarm}, one guarantees that all
poles of $e(\ccdot,v;z_0)$ are in $|\arg\l-\arg(v\!-\!z_0)|<\pi $). Then $K(\ccdot;z_0)$ is
holomorphic everywhere except $z_0$. Straightforward calculations give
\[
\weightE{z_0}{u_0^\pm}\cdot K(u_0^\pm;z_0)=\frac{4\theta_{z_0}}{\pi},\quad \mathrm{and}\quad
\weightE{z_0}{w_0^\pm}\cdot
K(w_0^\pm;z_0)=\frac{4\theta_{z_0}^*}{\pi}=\frac{4(\frac{1}{2}\pi-\theta_{z_0})}{\pi}.
\]
so $[\dopa K](z_0;z_0)\cdot \weightDS{z_0}= 1$. Scaling the plane, one may assume that
$\mesh=1$. As in Theorem~\ref{ThmKenyonHarm}, the integrand is exponentially small for
\mbox{intermediate $\l$}. One has
\[
e(\l,v;z_0)=\exp\left[\l(v\!-\!z_0)+O(|\l|^2|v\!-\!z_0|)\right],\qquad \l\to 0,
\]
and
\[
e(\l,v;z_0)=\frac{4\ol{\t}^2}{\l^2}\cdot\exp\lt[\frac{4\ol{(v\!-\!z_0)}}{\l} +
O\lt(\frac{|v\!-\!z_0|}{|\l|^2}\rt)\rt],\quad \l\to\iy,
\]
where $\t=e^{i\arg(u_0^{+}-u_0^{-})}$, if $v\in\G$, and $\t=e^{i\arg(w_0^{+}-w_0^{-})}$, if
$v\in\G^*$ ($\ol{\t}^2$ comes from the first factors (\ref{xEv0z0Def}) of $e(\l,v;z_0)$).
Summarizing, one arrives at
\[
K(u;z_0)=\frac{1}{\pi}\lt[\frac{1}{v\!-\!z_0} + \frac{\ol{\t}^2}{\ol{v\!-\!z_0}}\rt] +
O\lt(\frac{1}{|v\!-\!z_0|^2}\rt)= \frac{2}{\pi}\Pr\lt[\frac{1}{v\!-\!z_0};\ol{\t}\rt]+
O\lt(\frac{1}{|v\!-\!z_0|^2}\rt).
\]
Finally, $K(\ccdot,z_0)$ is unique due to Corollary~\ref{CorLipHarm}.
\end{proof}

\subsection{Proof of the discrete Harnack Lemma}
\label{SectAHarnack}

Below we recall the modification of R.~Duffin's arguments \cite{Duf53} given by U.~B\"ucking
\cite{Buck08}. For the next it is important that the remainder in (\ref{Gasympt}) is of order
$\mesh^2|u\!-\!u_0|^{-2}$ (and not just $\mesh|u\!-\!u_0|^{-1}$).

\begin{proposition}
\label{HmInDisc} Let $u_0\!\in\!\G$ and $R\ge \mesh$. Then
\[
\dhm{\mesh}{u_0}{\{a\}}{B^\mesh_\G(u_0,R)}\asymp \frac{\mesh}{R}\qquad \mathit{for\ all}\
a\in\pa B^\mesh_\G(u_0,R),
\]
i.e., $\const_1\cdot\,\mesh/R\le \dhm{\mesh}{u_0}{\{a\}}{\O^\mesh_\G} \le \,\const_2\cdot
\mesh/R$ for some positive absolute constants.
\end{proposition}

\begin{proof}
One has $R\le|a\!-\!u_0|\le R\!+\!2\mesh$ for all $a\in\pa B^\mesh_\G(u_0,R)$. Therefore,
(\ref{Gasympt}) gives
\[
\lt|G_{B^\mesh_\G(u_0,R)}(u;u_0) - \frac{1}{2\pi}\,\log\frac{|u\!-\!u_0|}{R}\rt|\le
\frac{\mesh}{\pi R} + \const\cdot \lt(\frac{\mesh^2}{|u\!-\!u_0|^2} + \frac{\mesh^2}{R^2}\rt)
\]
for all $u\ne u_0$. In particular, if $R/\mesh$ is large enough, then
\[
|G_{B^\mesh_\G(u_0,R)}(u;u_0)|\asymp \frac{\mesh}{R}\quad \mathrm{for\ all}\ u\in
B^\mesh_\G(u_0,R):\ R\!-\!5\mesh\le|u\!-\!u_0|\le R\!-\!3\mesh.
\]
Since Green's function is discrete harmonic and nonpositive near $\pa B^\mesh_\G(u_0,R)$, the
same holds true for all $a_{\mathrm{int}}:(a;a_{\mathrm{int}})\in\pa B^\mesh_\G(u_0,R)$. In
view of (\ref{oAsG}), this gives the result for sufficiently large $R/\mesh$. For small (i.e.
comparable to $\mesh$) radii $R$ the claim is trivial, since the random walk can reach $a$
starting from $u_0$ in a finite number of steps.
\end{proof}

\begin{proposition}[\bf mean value property]
Let $H:B^\mesh_\G(u_0,R)\to\R$ be a nonnegative discrete harmonic function. Then
\[
\lt|H(u_0)-\frac{1}{\pi R^2}\!\!\sum_{u\in \Int B^\mesh_\G(u_0,R)}\!\! H(u)\weightG{u}\rt|\le
\const\cdot \frac{\mesh H(u_0)}{R}.
\]
\end{proposition}

\begin{proof}
Let
\[
F(u):=G_\G(u;u_0) - \frac{\log R}{2\pi} + \frac{R^2-|u\!-\!u_0|^2}{4\pi R^2},\quad u\in
B^\mesh_\G=B^\mesh_\G(u_0,R).
\]
Note that $[\D^\mesh F](u)= -(\pi R^2)^{-1}$ for all $u\ne u_0$ (see
Lemma~\ref{ApproxLemma}(i)). Using (\ref{Gasympt}), it is easy to see that
$F(u)=O(\mesh^2/R^2)$, if $|u\!-\!R|\le \const\cdot\mesh$. The discrete Green's formula
(\ref{GreenFormula}) applied to $H$ and $F$ gives
\[
H(u_0)-\frac{1}{\pi R^2}\!\!\sum_{u\in \Int B^\mesh_\G}\! H(u)\weightG{u}\ = \sum_{a\in \pa
B^\mesh_\G}\tan\theta_{aa_{\mathrm{int}}}\cdot
[H(a_{\mathrm{int}})F^\pm(a)-H(a)F^\pm(a_{\mathrm{int}})],
\]
where the functions $F^\pm=F\pm \const\cdot\mesh^2/R^2$ are positive/negative, respectively,
near $\pa B^\mesh_\G(u_0,R)$. Using $H\ge 0$, one obtains
\[
-\const\cdot\frac{\mesh^2}{R^2}\sum_{a\in \pa B^\mesh_\G} H(a) \le H(u_0)-\frac{1}{\pi
R^2}\!\!\sum_{u\in \Int B^\mesh_\G}\! H(u)\weightG{u}\le
\const\cdot\frac{\mesh^2}{R^2}\sum_{a\in \pa B^\mesh_\G} H(a_{\mathrm{int}}),
\]
Both sums are comparable to $\mesh H(u_0)/R$ due to Proposition~\ref{HmInDisc}.
\end{proof}

\begin{proof}[\bf Proof of Proposition~\ref{PropHarnack}]
(i) Applying the mean value property for $B^\mesh_\G(u_0,R)$ and
$B^\mesh_\G(u_1,R\!-\!2\mesh)$ and taking into account that $H(u_1)\le \const\cdot H(u_0)$,
one obtains
\[
\pi R^2 H(u_0) - \pi (R\!-\!2\mesh)^2 H(u_1)\ \ \ =\!\!\!\!\!\!\! \sum_{u\in
B^\mesh_\G(u_0,R)\setminus B^\mesh_\G(u_1,R-2\mesh)} H(u)\weightG{u} + O(\mesh RH(u_0)).
\]
Proposition~\ref{HmInDisc} gives
\[
\sum_{u\in B^\mesh_\G(u_0,R)\setminus B^\mesh_\G(u_1,R\!-\!2\mesh)} H(u)\weightG{u}\ \
\asymp\!\!\!\!\! \sum_{u\in B^\mesh_\G(u_0,R)\setminus B^\mesh_\G(u_1,R-2\mesh)} \mesh^2H(u) \
=\ O(\mesh RH(u_0)),
\]
so $R^2\cdot(H(u_1)\!-\!H(u_0))=O(\mesh RH(u_0))$.

\smallskip

\noindent (ii) Let $u_1=v_0v_1v_2...v_{k-1}v_k=u_2$ be some path connecting $u_1$ and $u_2$
inside $B^\mesh_\G(u_0,r)$ (one can choose this path so that $k\le \const\cdot \mesh^{-1}r$).
Since $B^\mesh_\G(v_j,R\!-\!r)\ss B^\mesh_\G(u_0,R)$,
\[
\frac{H(u_2)}{H(u_1)}=\prod_{j=0}^{k-1}\frac{H(v_{j+1})}{H(v_j)}\le
\lt[1+\const\cdot\frac{\mesh}{R-r}\rt]^{\const\cdot {r}/{\mesh}}\!\!\!\!\le\
\exp\lt[\const\cdot\frac{r}{R-r}\rt].\qedhere
\]
\end{proof}

\end{document}